\documentclass[12pt]{amsart}
\usepackage[utf8]{inputenc}
\usepackage{etoolbox}

\makeatletter
 \def\@textbottom{\vskip \z@ \@plus 1pt}
 \let\@texttop\relax
\makeatother

\title{On the number of sets with small sumset}

\author{Dingyuan Liu}

\address{Dingyuan Liu \newline Institut für Algebra und Geometrie, Karlsruher Institut für Technologie, Englerstr. 2, D-76131 Karlsruhe, Germany}
\email{liu@mathe.berlin}

\author{Letícia Mattos}

\address{Letícia Mattos \newline Institut für Informatik, Universität Heidelberg, Im Neuenheimer Feld 205, D-69120 Heidelberg, Germany}
\email{mattos@uni-heidelberg.de}

\author{Tibor Szabó}

\address{Tibor Szabó \newline Institut für Mathematik, Freie Universität Berlin, Arnimallee 6, D-14195 Berlin, Germany}
\email{szabo@math.fu-berlin.de }

\thanks{TSz was partially funded by the DFG (German Research Foundation) under Germany’s Excellence Strategy -- The Berlin Mathematics Research Center MATH+ (EXC-2046/1, project ID 390685689)}

\usepackage{mathrsfs}
\usepackage{dsfont}
\usepackage{amsmath}
\usepackage{indentfirst}
\usepackage{mathtools}
\usepackage{amsfonts}
\usepackage{amsthm}
\usepackage{graphicx}
\usepackage[english]{babel}
\usepackage{fullpage}
\usepackage{enumerate}
\usepackage{tasks}
\usepackage{xcolor}
\usepackage{theoremref}
\usepackage{relsize}
\usepackage{setspace}
\usepackage{amssymb}
\usepackage{textcomp}
\usepackage[ruled]{algorithm2e} 	%Option: linesnumberedhidden
\SetKwInput{KwInput}{Input}                % Set the Input
\SetKwInput{KwOutput}{Output}   
\usepackage{caption}
\usepackage{hyperref}

\newtheorem{theorem}{Theorem}

\newtheorem{lemma}[theorem]{Lemma}

\newtheorem{conjecture}[theorem]{Conjecture}

\newtheorem{claim}[theorem]{Claim}

\numberwithin{theorem}{section}
\theoremstyle{remark}

\def\C{\mathcal{C}}
\def\D{\mathcal{D}}

\def\F{\mathcal{F}}

\def\HH{\mathcal{H}}
\def\I{\mathcal{I}}

\def\K{\mathcal{K}}
\def\LL{\mathcal{L}}

\def\Z{\mathbb{Z}}

\def\se{\subseteq}

\newcommand{\mychoose}[2]{\left({{#1}\atop#2}\right)}

\newcommand{\eps}{\varepsilon}
\newcommand{\sumrise}{\textsc{sumrise}}
\newcommand{\sunset}{\textsc{Sunset}}

\definecolor{lblue}{rgb}{0.5,0.5,1}

\newcommand{\eq}[1]{\begin{equation}\label{eq:#1}}
	\newcommand{\eqe}{\end{equation}}

\begin{document}

\begin{abstract}
	We investigate subsets with small sumset in arbitrary abelian groups.
	For an abelian group $G$ and an $n$-element subset $Y \subseteq G$ we show that if $m \ll s^2/(\log n)^2$, then the number of subsets $A \subseteq Y$ with $|A| = s$ and $|A + A| \leq m$ is at most
    \[2^{o(s)}\binom{\frac{m+\beta}{2}}{s},\]
    where $\beta$ is the size of the largest subgroup of $G$ of size at most $\left(1+o(1)\right)m$.
	This bound is sharp for $\mathbb{Z}$ and many other groups.
	Our result improves the one of Campos and nearly bridges the remaining gap in a conjecture of Alon, Balogh, Morris, and Samotij.

	We also explore the behaviour of uniformly chosen random sets $A \subseteq \{1,\ldots,n\}$ with $|A| = s$ and $|A + A| \leq m$.
	Under the same assumption that $m \ll s^2/(\log n)^2$, we show that with high probability there exists an arithmetic progression $P \subseteq \mathbb{Z}$ of size at most $m/2 + o(m)$ containing all but $o(s)$ elements of $A$. Analogous results are obtained for asymmetric sumsets, improving results by Campos, Coulson, Serra, and Wötzel.
  
    The main tool behind our results is a more efficient container-type theorem developed for sets with small sumset, which gives an essentially optimal collection of containers. The proof of this combines an adapted hypergraph container lemma, that caters to the asymmetric setup as well, with a novel ``preprocessing'' graph container lemma, which allows the hypergraph container lemma to be called upon significantly less times than was necessary before.

\end{abstract}
\vspace{-1em}
\maketitle

\vspace{-30pt}
\section{Introduction}

	Let $(G,+)$ be an abelian group. The sumset of a set $A \subseteq G$ is defined as 
	\[A+A = \{a_1 + a_2 : a_1, a_2 \in A\}.\]
	A key question in additive combinatorics is: what can the size of the sumset $A+A$ reveal about the structure of $A$?
	Freiman's theorem~\cite{freiman1959addition} (see also~\cite{green2007freiman,schoen2011near}) states that for $A \subseteq \mathbb{Z}$, if the sumset $|A+A|$ is ``small'', $A$ can be found in a generalised arithmetic progression with bounded volume and dimension, based on the doubling factor $|A+A|/|A|$. 
	Green and Ruzsa~\cite{green2007freiman} extended this to all abelian groups. 
	In a major breakthrough, Gowers, Green, Manners, and Tao~\cite{gowers2023conjecture} proved Marton's conjecture,
	showing that for groups with bounded torsion, a set $A$ with a small sumset can be covered by a polynomial number of translates (depending on the doubling) of a subgroup of size $|A|$.

	For integers $m,s \in \mathbb{N}_{\ge 1}$ and subset $Y\subseteq G$, let $\F_Y(m,s)$ denote the family of subsets $A \se Y$ of size $|A| = s$ such that $|A+A| \le m$. 
	The family of {\em all} subsets with sumset of size at most $m$ is denoted by $\F_Y(m) = \cup_{s=0}^{m} \F_Y(m,s)$.

\subsection*{The refined problem.}
	The study of the size of the refined family $\F_{Y}(m,s)$ was initiated by Green~\cite{green2005counting} for $\mathbb{Z}_n$,
	who gave an  upper bound for every range of the doubling factor $m/s$ and used it as a tool to bound the clique number of random Cayley graphs. 
	Subsequently Alon, Balogh, Morris, and Samotij~\cite{alon2014refinement}, while studying a refinement of the Erd\H os--Cameron conjecture, employed families of subsets of $[n] \subseteq \mathbb{Z}$ with small sumset.
    They noted that for any $s$-subset $A \se [n]$ which is contained in an arithmetic progression $P$ of size $\lfloor \frac{m+1}{2} \rfloor$, one has $|A+A| \leq |P+P| = 2|P| -1 \leq m$, and hence 
    \begin{align*}
        |\F_Y(m,s)| = \Omega\left( \frac{n^2s}{m^2} \right) \mychoose{ \lfloor \frac{m+1}{2}\rfloor }{s}
    \end{align*}
    whenever $6s \le 0.99m$ and $\frac{m+1}{2} \le 0.99n$.
    As this lower bound has neither been explicitly stated nor proved in the literature, we present the details in Appendix~\ref{appendix:lower-bound}.
    This motivated them to conjecture the following. 
    Unless otherwise specified, throughout this paper we default that all logarithms are in base $2$. For all $x \in \mathbb{R}$ and $s \in \mathbb{N}_{\ge 0}$ define $\binom{x}{s}$ to be $x(x-1)\cdots(x-s+1)/s!$ if $x\ge s$ and $0$ if $x<s$.

	\begin{conjecture}~\cite[Conjecture 1.5]{alon2014refinement} \label{conj:ABMS}
		For every $\delta > 0$, there exists $C>0$ such that the following holds. 
		If $m \le \frac{s^2}{C}$ and $s \ge C \log n$, then there are at most
		\[ 2^{\delta s} \binom{m/2}{s} \]
		sets $A \se [n]$ with $|A| = s$ and $|A+A| \le m$.
	\end{conjecture}

	A construction by Morris, described in~\cite{campos2020number}, shows that the conjecture is false when the doubling factor $m/s$ is too large, more precisely, when $m/s \gg s(\log n)^{-1}$.
	Observe that if $m/s = O\big(s(\log n)^{-1}\big)$, then this automatically implies that $s$ should be at least of order $\log n$, since trivially $m\geq s$.

	Conjecture~\ref{conj:ABMS} holds trivially whenever the doubling factor $m/s$ is strictly less than $2$, since $|A+A| \geq 2|A| -1$ for every set $A \subseteq [n]$. 
	Green and Morris~\cite{green2016counting} proved the conjectured bound whenever the doubling factor $m/s$ is upper bounded by a constant.
   In the same range of $m/s$, Campos, Collares, Morris, Morrison, and Souza~\cite{campos2022typical} even managed to determine the size of the family $\F_{[n]}(m,s)$ up to a constant factor, given that $(\log n)^4 < s = o(n)$.    
   Meanwhile, Campos~\cite{campos2020number} improved the Green--Morris theorem in a different direction: he extended the range of validity of the Alon--Balogh--Morris--Samotij conjecture up to $m/s \ll s(\log n)^{-3}$. 
   This leaves a $(\log n)^2$-factor gap between the range of the doubling factor for which Campos' result holds and the range for which the conjecture is false.
   
In our first main theorem we improve on Campos' theorem and reduce, from $(\log n)^2$ to $\log n$, the gap between the ranges of Morris' counterexample and when the conjecture is true. Furthermore, similar to Campos~\cite{campos2020number}, we also extend our result to arbitrary abelian groups.
   As in~\cite{campos2020number}, our bound depends on the quantity 
   \begin{align}\label{eq:defn-beta}
		\beta_G(m)\coloneqq \max\{|H|: H \le G \text{ and } |H| \le m\}.
	\end{align}
	In words, $\beta_G(m)$ is the size of the largest subgroup of $G$ whose size is at most $m$.
    Note that $\beta_{\mathbb{Z}}(m) = 1$ for every $m \in \mathbb{N}_{\ge 1}$.

	\begin{theorem}\label{thm:symmetric-refined} 
		Let $n, m=m(n), s=s(n) \in \mathbb{N}_{\ge 1}$ be integers such that $2s \le m \ll \frac{s^2}{(\log n)^2}$ as $n \rightarrow \infty$. 
		Let $G$ be an arbitrary abelian group and $Y$ be an $n$-element subset of $G$\footnote{For a meaningful statement we actually need to speak about sequences of groups $G_n$ and subsets $Y_n \subseteq G_n$, where $|Y_n| = n \in \mathbb{N}_{\ge 1}$.}.
		Then, the number of subsets $A \in \binom{Y}{s}$ such that $|A+A| \le m$ is at most
		\[2^{o(s)} \mychoose{\frac{m+\beta}{2}}{s},\]
        where $\beta = \beta_G\left(m+o(m)\right)$.
	\end{theorem}
    A construction pointed out in~\cite{campos2020number} shows that our theorem is tight for many abelian groups (up to the $2^{o(s)}$ factor). The main tool behind our result is a more efficient container-type theorem (see Theorem~\ref{thm:simple-containers}), developed for sets with small sumset, which gives an essentially optimal collection of containers. The proof of this combines an adapted hypergraph container lemma following a novel ``preprocessing'' graph container lemma, which reduces the number of times the hypergraph container lemma needs to be applied by a logarithmic factor compared to previous methods.

    Campos~\cite{campos2020number} also studied the typical structure of sets in $\F_{[n]}(m,s)$ for the classical case where $G = \Z$ and $Y = [n]$.
	Namely, he showed that for $m \ll \frac{s^2}{(\log n)^3}$ a uniformly chosen random element $A$ of the family $\F_{[n]}(m,s)$, with high probability, 
	is almost completely (up to $o(s)$ elements) contained in an arithmetic progression of size $m/2 +o(m)$.  
	For the range where $m/s =O(1)$ and $(\log n)^4 < s = o(n)$, 
	a much more accurate structural result was achieved in~\cite{campos2022typical}. 
	More precisely, in~\cite{campos2022typical} they prove that with probability $1-\varepsilon$ the random subset $A$ is {\em fully} contained in an arithmetic progression of size $m/2+O_{\varepsilon, m/s}(1)$ and they also show that the $O(1)$ term is necessary. 

	In our next theorem, we extend the range of validity of Campos' structural theorem by a factor of $\log n$. 
	
	\begin{theorem}\label{thm:symmetric-typical-structure}
		Let $n, m=m(n), s=s(n) \in \mathbb{N}_{\ge 1}$ be integers such that $m \ll \frac{s^2}{(\log n)^2}$ as $n \rightarrow \infty$. If $A$ is chosen uniformly at random from $\F_{[n]}(m,s)$ then with high probability there exists an arithmetic progression $P$ of size $|P| \le m/2 + o(m)$ such that $|A \setminus P| = o(s)$.
	\end{theorem}

Note that, due to the presence of the error term $o(s)$, this does not directly imply our bound in Theorem~\ref{thm:symmetric-refined} on the counting problem.
Unlike Conjecture~\ref{conj:ABMS}, it may be true that Theorem~\ref{thm:symmetric-typical-structure} holds up to $m \le s^2/C$, for some large constant $C>0$, as this does not conflict with any known constructions that provide lower bounds for $|\mathcal{F}_{[n]}(m,s)|$.
In any case, we are now only a $(\log n)^2$ factor away from being best possible.

\subsection*{The unrefined problem.} Next we study the ``unrefined" family $\F_Y(m) = \cup_{s=0}^{m} \F_Y(m,s)$, where the size restriction is not present. Here our methods are suitable to provide satisfyingly sharp bounds in the whole range of the parameters for any abelian group $G$ and an $n$-subset $Y \se G$.  

A construction pointed out by Campos~\cite{campos2020number} shows that the size of $\F_Y(m)$ can be as large as $2^{\frac{m+\beta}{2}}$, where $\beta = \beta_G(m)$. 
If $m \gg (\log n)^4$, then one can show that this lower bound is tight up to a factor $2^{o(m)}$ by summing up the upper bounds on $|F_Y(m,s)|$ given by Theorem~\ref{thm:symmetric-refined} for $s \gg \sqrt{m}\log n$ and the number of all sets of size $O(\sqrt{m}\log n)$.
This can be extended down to $m \gg (\log n)^3$ by employing directly a container theorem of Campos (see Theorem 4.2 in~\cite{campos2020number}). Moreover, the same approach gives $|\F_Y(m)|\le 2^{O(\sqrt{m}(\log n)^{3/2})}$ in the complement range $m\ll (\log n)^3$.

In our next theorem we prove that, on the one hand, the upper bound $2^{\frac{m+\beta}{2}+o(m)}$ on $|\F_Y(m)|$, which is sharp up to the factor $2^{o(m)}$, is valid for $m \gg (\log n)^2$. 
On the other hand, we are able to couple this by an improvement of the above in the complementary range $m \ll (\log n)^2$. Namely, we show that in this range the lower bound provided by the trivial construction of taking every subset of size at most $\lfloor \sqrt{2m}\rfloor -1$ is tight up to a constant factor $\sqrt{2}$ in the exponent.   
Recall that all logarithms are in base $2$.
\begin{theorem}\label{thm:unrefined-counting}
	Let $n, m=m(n) \in \mathbb{N}_{\ge 1}$ and set $\alpha(n) = \left (\frac{\sqrt{m}}{\log n}\right )^{1/2}$. 
	Let $G$ be an abelian group and $Y$ be an $n$-element subset of $G$.
	Then, we have
	\[|\F_Y(m)| \le \begin{cases} 2^{(2+3\alpha)\sqrt{m}\log n + m} & \text{ if } m \le (\log n)^2, \\ 2^{\frac{m+\beta}{2} + 2^{13} \alpha^{-2/3} m} & \text{ if } m > (\log n)^2,
	 \end{cases} \]
  where $\beta=\beta_{G}\left(m+\alpha^{-2/3}m\right)$. In particular, if $m \ll (\log n)^2$, then $|\F_Y(m)| \le n^{(2+o(1))\sqrt{m}}$, 
	 and if $m \gg (\log n)^2$, then $|\F_Y(m)| \le 2^{\frac{m+\beta}{2} + o(m)}$.
\end{theorem}

We note that when $G = \mathbb{Z}$ and $Y = [n]$, similar bounds to those in Theorem~\ref{thm:unrefined-counting} can also be deduced by combining the results of Green~\cite{green2005counting} (see Proposition 23 in~\cite{green2005counting}) and Green and Morris~\cite{green2016counting} (see Theorem 1.1 in~\cite{green2016counting}).
The cutoff $(\log n)^2$ arises naturally in this context because $|\mathcal{F}_{[n]}(m)| \ge \max 
\left \{ \binom{n}{\lfloor \sqrt{2m} \rfloor -1}, 2^{\frac{m}{2}} \right \}$.
Specifically, the binomial term dominates when $m \le c(\log n)^2$, while the exponential term takes over for $m>c(\log n)^2$, for some constant $c>0$.

\subsection*{The asymmetric setup}

In 2023, Campos, Coulson, Serra, and Wötzel~\cite{campos2023typical} investigated an asymmetric version of the counting problem. 
	Given integers $s,s',m,n$, an abelian group $G$ and an $n$-element subset $Y \se G$,
	let $\F_Y(m,s,s')$ be the set of pairs $(A,B) \in \binom{Y}{s} \times \binom{Y}{s'}$  such that $|A+B|\le m$.

 To grasp the behavior of $|\F_Y(m,s,s')|$, we again first consider the case when $G = \mathbb{Z}$ and $Y = [n]$. A simple constructive lower bound is given by the pairs $(A,B)$ where
$A$ is an $s$-subset in the interval $[x]$ and $B$ is an $s'$-subset in the interval $[m-x]$. Assuming without loss of generality that $s\leq s'$ and setting $x = \left\lfloor \frac{sm}{s+s'}\right\rfloor$, we obtain that
	\begin{align}\label{eq:lower-bound-F-asym}
		|\F_{[n]}(m,s,s')| \ge \mychoose{x}{s} \mychoose{m-x}{s'}.
	\end{align}

	 In~\cite{campos2023typical} it is shown that if $s' = \Theta(s)$ and $s+s' \le m \ll \frac{s^2}{(\log n)^3}$, then \eqref{eq:lower-bound-F-asym} is best possible up to a factor $2^{o(s)}$. In the case where $m = s+s' + o(s)$ this already follows from a structural theorem of Lev and Smeliansky~\cite[Theorem 2]{lev1995addition} for $G = \mathbb{Z}$ and $Y = [n]$. 
  
Our next theorem extends by a $\log n$-factor the range of validity of the result of  Campos, Coulson, Serra, and Wötzel~\cite{campos2023typical}.
Similarly to~\cite{campos2023typical}, our result also generalises to arbitrary abelian groups.

\begin{theorem}\label{thm:asymmetric-counting}
	Let $n, m=m(n), s=s(n), s'=s'(n) \in \mathbb{N}_{\ge 1}$ be integers such that $ s' = \Theta (s)$ and $s+s' \le m \ll \frac{s^2}{(\log n)^2}$ as $n\rightarrow \infty$. Let $G$ be an arbitrary abelian group and $Y$ be an $n$-element subset of $G$.
Then the number of pairs $(A,B) \in \binom{Y}{s} \times \binom{Y}{s'}$ of subsets  with $|A+B| \le m$ is at most 
		$$2^{o(s)} \mychoose{\frac{(m+\beta)s}{s+s'}}{s} \mychoose{\frac{(m+\beta)s'}{s+s'}}{s'},$$
  where $\beta = \beta_G\left(m+o(m)\right)$.
\end{theorem}

We note that one can adapt the construction in~\cite{campos2020number} to show that the upper bound from Theorem~\ref{thm:asymmetric-counting} cannot hold if 
$m \gg s^2/\log n$. In other words, similar to Theorem~\ref{thm:symmetric-refined}, we are a $\log n$-factor away from establishing the precise region of validity. 
Nevertheless note that Theorem~\ref{thm:asymmetric-counting} about the asymmetric setup does not imply much about the symmetric setup in Theorem~\ref{thm:symmetric-refined}, as there are much more asymmetric pairs $(A,B)$ than sets $A$ with sumset bounded by $m$.

A couple of constructions from~\cite{campos2023typical} are also relevant to our Theorem~\ref{thm:asymmetric-counting}. They imply that some error factor in the upper bound as well as some restriction between the proportion of $s$ and $s'$ are both necessary.

	In the case where $G = \Z$ and $Y = [n]$, the typical structure of sets in  $\F_{[n]}(m,s,s')$ was also studied in~\cite{campos2023typical}.
Namely, they showed that if $s' = \Theta(s)$ and $s+s'\le m \ll \frac{s^2}{(\log n)^{3}}$, then for a uniformly chosen random element $(A,B)$ of the family $\F_{[n]}(m,s,s')$, with high probability, $A$ and $B$ are almost completely (up to $o(s)$ elements) contained in some arithmetic progressions $P_A$ of size $\frac{sm}{s+s'} +o(m)$ and $P_B$ of size $\frac{s'm}{s+s'} +o(m)$, respectively.  
 Again, in the case where $m = s+s' + o(s)$, this is a consequence of a result of Lev and Smeliansky~\cite[Theorem 2]{lev1995addition}. 

	In our next theorem we extend the range of validity of the structural theorem of Campos, Coulson, Serra, and Wötzel~\cite{campos2023typical} by a factor of $\log n$. 

\begin{theorem}\label{thm:asymmetric-typical-structure}
Let $n, m=m(n), s=s(n), s'=s'(n) \in \mathbb{N}_{\ge 1}$ be integers such that $ s' = \Theta (s)$ and $s+s' \le m \ll \frac{s^2}{(\log n)^2}$ as $n\rightarrow \infty$.
If $(A,B) \in 2^{[n]} \times 2^{[n]}$ is a uniformly chosen random pair with $|A| = s$, $|B| = s'$ and $|A+B| \leq m$, then with high probability there exist arithmetic progressions $P_A, P_B \subseteq [n]$ of sizes
		\[ |P_A| \le \dfrac{sm}{s+s'} +o(m) \quad \text{and} \quad |P_B| \le \dfrac{s'm}{s+s'} +o(m)\]
		such that $|A \setminus P_A| + |B \setminus P_B| = o(s)$.
	\end{theorem}

Note, again, that the presence of the $o(s)$ term in the size of the arithmetic progression prevents to derive the bound in Theorem~\ref{thm:asymmetric-counting} directly. 
Moreover, since the pair $(A,B)$ is chosen uniformly at random in $\F_{[n]}(m,s,s')$, which is a much larger family than $\F_{[n]}(m)$, Theorem~\ref{thm:asymmetric-typical-structure} has no direct implications for Theorem~\ref{thm:symmetric-typical-structure}.

\subsection*{The asymmetric setup --- unrefined.} 
Our methods can also be used to give bounds on the asymmetric version of the unrefined family $\F_Y(m)$.
	For a group $G$ and a subset $Y \se G$, let $\F_Y^{\mathrm{asym}}(m)$ be the family of all pairs $(A,B) \in 2^Y \times 2^Y$ such that $|A+B| \le m$.
	By taking $A$ to be a set of size one and $B$ to be any set of size at most $m$ in $Y$, it follows that 
	$|\F_Y^{\mathrm{asym}}(m)| \ge n\binom{n}{\le m}$,
    where $\binom{n}{\le m}\coloneqq \sum_{k=0}^{m}\binom{n}{k}$.
    Our next theorem shows that this is best possible up to an error factor of $\binom{n}{\le m}^{o(1)}$  
 when $1 \ll m \ll n$.
 
	We also show that the biggest contribution to $\F_Y^{\mathrm{asym}}(m)$ comes from very asymmetric pairs. Namely, if we additionally assume a lower bound $\sqrt{m}$ on the sizes of the sets $A$ and $B$, then we obtain bounds that are analogous to the ones in Theorem~\ref{thm:unrefined-counting}. 
    In particular, these will show that the contribution of these pairs to $\F_Y^{\mathrm{asym}}(m)$ is negligible.
    To state these results we define $\F_{Y}^{\mathrm{asym}}(m, \ge s)$ to be the family of pairs $(A,B) \in \F_Y^{\mathrm{asym}}$ such that $|A|, |B| \ge s$.

	\begin{theorem}\label{thm:unrefined-counting-part-2}
            Let $n, m=m(n) \in \mathbb{N}_{\ge 1}$ be integers such that $m \le 2^{-4}n$.
            Let $G$ be an abelian group and $Y$ be an $n$-element subset of $G$.	Then, we have
			\[|\F_Y^{\mathrm{asym}}(m)| \le n^{6\sqrt{m}}\binom{n}{\le m}.\]
Furthermore, let $s=s(n)\ge\sqrt{m}$ be an integer, then
			\[|\F_{Y}^{\mathrm{asym}}(m,\ge s)| \le \begin{cases} n^{(4+6\alpha)\sqrt{m}} & \text{ if } m \le (\log n)^2, \\ 2^{(1+2^{13}\alpha^{-2/3})(m+\beta)} & \text{ if } m > (\log n)^2,
	 		\end{cases} \]
    where $\alpha(n)=\left(\frac{\sqrt{m}}{\log{n}}\right)^{1/2}$ and $\beta = \beta_{G}\left(m+\alpha^{-2/3}m\right)$.
	\end{theorem}

    We note that Alon, Granville, and Ubis~\cite{alon2010sumsets} explored a related problem, providing exact asymptotics for the number of distinct sumsets $A+B$ in $\mathbb{F}_p$, where $|A|, |B| \gg 1$ as $p$ tends to infinity, and without imposing any size constraints on the sumset.

\subsection*{Our asymmetric container theorem}
The main tool behind the proof of all our theorems is the container theorem below. 

 For a subset $Y$ in an abelian group $G$ and subsets $U_0 \se Y+Y$ and $U_1, U_2 \se Y$, we define $\HH(U_0,U_1,U_2)$ to be the tripartite 3-uniform hypergraph with parts $U_0, U_1$ and $U_2$, in which vertices $a \in U_1$, $b \in U_2$ and $c \in U_0$ are connected if and only if $c = a+b$.
 For simplicity, we also denote $(A+B)^c \coloneqq (Y+Y) \setminus (A+B)$.

	\begin{theorem}\label{thm:simple-containers}
            Let $n,m \in \mathbb{N}_{\ge 1}$ and $0<\eps<1/4$.
            Let $G$ be an arbitrary abelian group, $Y$ be an $n$-element subset of $G$ and $\beta = \beta_G((1+4\eps)m)$.
		There exists a collection $\C \se 2^{Y+Y} \times 2^Y \times 2^Y$ of size at most $n^{2^{7}\eps^{-2}\sqrt{m}}$ such that the following holds. For every pair $(A,B) \in 2^Y \times 2^Y$ with $|A|\ge \sqrt{m},|B| \ge (1+\eps^{-2})\sqrt{m}$ and $|A+B|\le m$ there exists $(C_{0},C_{1},C_{2}) \in \C$ such that 
		\begin{enumerate}
			\item $(A+B)^c \se C_0$, $A \se C_1$ and $B \se C_2$;
            \vspace{1mm}
			\item $|C_{1}|+|C_2|\le (1+2\eps)(m + \beta)$.
		\end{enumerate}
        Additionally, we have $|C_{1}|,|C_2| \le m/4$ or $e(\HH(C_0,C_1,C_2)) < \eps^2|C_{1}||C_2|$.
	\end{theorem}

	As usual in container-type theorems, our container theorem gives us a ``small'' collection $\C$ of containers such that every set in the family we want to count is inside some container, and the size of every container is significantly smaller than the size of the ground set.
	The main difference between our container theorem and the previous ones is the size of the collection.  
	While similar collections in~\cite[Theorem 4.2]{campos2020number} and~\cite[Theorem 2.2]{campos2023typical} have size $\exp\left(O\left(\eps^{-2}\sqrt{m}(\log n)^{3/2}\right)\right)$, our collection has size $\exp\left(O\left(\eps^{-2}\sqrt{m}\log n\right)\right)$.

    Next we briefly explain the idea behind the improvement. In Campos' proof~\cite{campos2020number}, he defined an auxiliary hypergraph so that every $A\subseteq{Y}$ with $\lvert{A+A}\rvert\leq{m}$ corresponds to a ``special'' independent set, and then he obtained a collection of containers for such independent sets by iteratively applying a hypergraph container lemma until the container size reaches $(1+2\varepsilon)(m+\beta)/2$. However, this process requires $\Omega(\varepsilon^{-2}\log{n})$ iterations in the worst case. 
    The main novelty in our approach is that we first use a bipartite graph container algorithm (see Lemma~\ref{lemma:first-container}) to obtain an initial collection of $n^{O(\sqrt{m})}$ containers of size $O(m)$, after which it is sufficient to apply our adapted hypergraph container algorithm (see Lemma~\ref{lemma:container-shrinking}) only $O(\varepsilon^{-2})$ times to reduce the container size. This results in a much smaller family of containers. The feature of our graph container algorithm is that we only traverse a certain part of vertices instead of the entire vertex set. This allows us to control the intersection size of every container and this certain part, which cannot be achieved by applying existing graph container lemmas.

    Our hypergraph container algorithm, which we shall apply in the second phase, is designed for a specific family of hypergraphs that encapsulates the asymmetric setup as well. It provides containers for all pairs $(A,B)$ of subsets satisfying $|A+B|\leq{m}$ and reduces the container size more efficiently than the corresponding lemma in~\cite{campos2023typical}. We note that a symmetric form of Theorem~\ref{thm:simple-containers} considering only the pairs $(A,A)$ would be sufficient to derive our symmetric counting and structural results (Theorems~\ref{thm:symmetric-refined},~\ref{thm:symmetric-typical-structure} and~\ref{thm:unrefined-counting}), and to prove this symmetric form one could also employ Campos' hypergraph container lemma~\cite[Theorem 2.2]{campos2020number} in the second phase of our approach. 
    We also note that if the hypergraph container lemma of Campos, Coulson, Serra, and Wötzel~\cite[Theorem 2.1]{campos2023typical} is applied in the second phase of our approach, then one could obtain a weaker form of Theorem~\ref{thm:simple-containers}, with a larger collection of containers and without the last additional property. 
    This weaker form would be sufficient to deduce our asymmetric counting results (Theorems~\ref{thm:asymmetric-counting}~and~\ref{thm:asymmetric-typical-structure}) with a worse error term.  
    For the asymmetric structural result (Theorem~\ref{thm:unrefined-counting-part-2}), we need the full strength of Theorem~\ref{thm:simple-containers}.
    That is, we not only need the collection of containers to be small, but we also need to know their internal structure revealed by the last additional property. The hypergraph container lemma in~\cite{campos2023typical} cannot achieve both goals at the same time, even with our graph container algorithm applied beforehand. In contrast, our algorithm in Lemma~\ref{lemma:container-shrinking} specifies which part of the container to shrink every time, and this proves to be crucial for reducing the container size and keeping track of its internal structure simultaneously in the second phase.

It is also worth noting that Theorem~\ref{thm:simple-containers} is essentially optimal in the sense that neither the container size nor the number of containers can be significantly improved. The upper bound on the container size cannot be further improved, because the construction in~\cite{campos2020number} points out that there exist sets $A,B\subseteq{G}$ of size $\frac{m+\beta}{2}$ such that $\left\lvert{A+B}\right\rvert\leq{m}$. Namely, the corresponding $(C_{0},C_{1},C_{2})$ for such $(A,B)$ must satisfy $|C_{1}|+|C_2|\geq{m + \beta}$. On the other hand, the number of containers cannot be improved to $n^{\sqrt{m}}$. Otherwise when $m\ll(\log{n})^{2}$, this would yield an upper bound $2^{\left(1+o(1)\right)\sqrt{m}\log{n}}$ on the number of $A\subseteq{Y}$ with $\left\lvert{A+A}\right\rvert\leq{m}$, which contradicts the trivial lower bound $\binom{n}{<\sqrt{2m}}$.

	The rest of this paper is divided as follows.
	In Section~\ref{sec:proof-idea} we give a brief introduction on the container method and prove Theorem~\ref{thm:simple-containers};
	in Section~\ref{sec:first-container}, we introduce our asymmetric graph container algorithm and prove Lemma~\ref{lemma:first-container}, which is an unrefined version of Theorem~\ref{thm:simple-containers};
	in Section~\ref{sec:second-container}, we introduce our hypergraph container algorithm and derive our so-called container shrinking lemma (see Lemma~\ref{lemma:container-shrinking}) from it; in Section~\ref{sec:counting}, we derive all our counting theorems from Theorem~\ref{thm:simple-containers}, namely Theorems~\ref{thm:symmetric-refined},~\ref{thm:unrefined-counting},~\ref{thm:asymmetric-counting} and~\ref{thm:unrefined-counting-part-2}; and finally in Section~\ref{sec:typical-structure}, we prove our typical structure theorems, namely Theorems~\ref{thm:symmetric-typical-structure} and~\ref{thm:asymmetric-typical-structure}.

    \section{Container theorems and proof idea} \label{sec:proof-idea}
	In this section we explain the idea behind the statement and the proof of Theorem~\ref{thm:simple-containers}, which uses the so-called container method. A container is a subset of vertices in a (hyper)graph, which potentially contains many independent sets.
	The history of container-type theorems traces back to the works of Kleitman and Winston~\cite{kleitman1982number} and Sapozhenko~\cite{sapozhenko2001independent,sapozhenko2003cameron,sapozhenko2005systems,sapozhenko2007number}, who developed a graph container algorithm to bound the number of independent sets in certain auxiliary graphs.
	Later, Balogh, Morris, and Samotij~\cite{balogh2015independent} and, independently, Saxton and Thomason~\cite{saxton2015hypergraph} 
	developed a general hypergraph container algorithm to count families without forbidden configurations.
    Balogh and Samotij~\cite{baloghsamotij} proved an efficient version which is suitable for hypergraphs with large uniformities.
	More recently, Morris, Samotij, and Saxton~\cite{morris2018asymmetric} introduced a refined version of the container method that takes into account the asymmetry of certain structures.
	This refinement turns out to be extremely useful in additive combinatorics, and a variant of it was recently used by Campos~\cite{campos2020number} to bound the number of sets with small sumset under certain conditions.
	The containers in~\cite{campos2020number} were also used in~\cite{campos2022typical,campos2023typical} to determine the typical structure of sets with small sumset.

	For a graph $H$ on vertex set $[n]$, let $\I_j(H)$ be the collection of independent sets of size $j$ in $H$.
	The rough idea behind graph containers is that if $j$ is ``large'' and $H$ is roughly regular, then the independent sets in $\I_j(H)$ must somehow be clustered together.
	That is, it is possible to obtain a ``small'' collection $\C \se 2^{[n]}$ of ``small'' sets such that every set $I$ in $\I_j(H)$ is contained in some set in $\C$.
	This allows us to go from the  trivial upper bound of $\binom{n}{j}$ on the size of $\I_j(H)$ to $\sum_{C \in \C} \binom{|C|}{j}$. This represents an exponential improvement if $|\C| = 2^{o(j)}$ and $|C| < (1-\delta)n$ for all $C \in \C$, for some constant $\delta >0$.

	The Kleitman--Winston container algorithm to obtain such a collection $\C$ is quite simple.
	Before implementing it, we fix an ordering of the vertices $v_1,\ldots,v_n$ in $V(H)$ for each graph $H$.
	As an input, the algorithm receives a graph $H$, a parameter $s$, and an independent set $I$ of size at least $s$.
	Initially, it sets $S = \emptyset$ (the so-called fingerprint of $I$) and
	$U = V(H)$ (the set of active vertices).
	While $|S| < s$, the algorithm selects a vertex $v \in U$ of maximum degree in $H[U]$ (breaking ties according to the pre-fixed order),
	and asks whether $v$ is in $I$. If it is, then we place $v$ in $S$ and remove $v$ and its neighbours from $U$. Otherwise, we remove $v$ from $U$.
	The algorithm stops after $|S| = s$ and outputs $S(I)=S$ and $U(I)=U$. 

It is immediate that $I \se U \cup S$ and one can check that if two independent sets $I$ and $I'$ output the same fingerprint $S$, then they also output the same $U$.
	Therefore, the collection $\C$ given by $\{ U(I) \cup S(I): I\subseteq V(H)\text{ an independent set and } |I| \ge s \}$ is a collection of at most $\binom{n}{s}$ containers.
	Now, assuming that $H$ is roughly $d$-regular, if we choose $s \approx 2\delta n/d$, then we would expect that every time we place a vertex in $S$, we remove roughly at least $d/2+1$ vertices from $U$.
	Therefore, we would expect that $|U|+|S| \le (1-\delta)n$ after the algorithm stops. Moreover, if $d\gg1$, namely, $s\ll n$, then the number of containers would be sufficiently small and hence $\C$ is the desired collection.

	Given this effective tool for counting independent sets in graphs, it is natural to consider reformulating our counting problem in terms of counting independent sets in some auxiliary graph. 
	Let $G$ be an arbitrary abelian group and $Y$ be an arbitrary set in $G$.
	For subsets $U_0 \se Y+Y$ and $F, U_1 \se Y$, we define $H^{F}(U_0,U_1)$ to be the bipartite graph with parts $U_0$ and $U_1$ in which two vertices $a \in U_1$ and $b \in U_0$ are connected by an edge if and only if $b-a \in F$.
	Then, one can see that for every $A \se Y$, the set $(A+F)^c \cup A$ is an independent set in $H^{F}(U_0,U_1)$, where $(A+F)^c \coloneqq U_0 \setminus (A+F)$.

    \subsection*{Our first container lemma}
	Our main improvement in this paper comes from the use of an asymmetric graph container algorithm, which we call \sumrise, to count independent sets in the graphs $H^{F}(U_0,U_1)$.
	This is one of the main differences between our work and previous work, which reformulated the counting problem as counting independent sets in 3-uniform hypergraphs instead of graphs.

	For $\lambda,s,m \in \mathbb{N}_{\ge 1}$, define
	\begin{align}\label{eq:defn-L-m}
		\LL_m(s,\lambda) \coloneqq \left \{ (A,F) \in 2^{Y} \times 2^{Y} 
		\hspace*{1.5mm}
	: \, |A| \ge s, |F| \ge \lambda \text{ and } |A+F| \le m \right \}.
	\end{align}
	Our first container lemma can be thought of as an unrefined version of Theorem~\ref{thm:simple-containers}.

\begin{lemma}\label{lemma:first-container}
	Let $m, n, \lambda, s \in \mathbb{N}_{\ge 1}$ and $\delta > 0$ be such that $\delta \lambda s \ge (1+\delta)m$.
	Let $Y$ be an $n$-element subset of an arbitrary abelian group $G$.
	Then, there exists a collection $\C \se 2^{Y}$ such that the following holds:
	\begin{enumerate}
		\item $|\C| \le \displaystyle \binom{n}{s} \binom{n}{\lambda}$;
        \vspace{6pt}
		\item $|C| \le (1+\delta)m$ for every $C \in \C$;
        \vspace{6pt}
		\item For every $(A,F)\in \LL_m(s,\lambda)$, there exists $C \in \C$ such that $A \se C$.
	\end{enumerate}
\end{lemma}

	We prove Lemma~\ref{lemma:first-container} in Section~\ref{sec:first-container}.
        Let $V_0$ be a copy of $Y+Y$ and $V_1$ be a copy of $Y$. 
	The idea is to apply our~\sumrise~algorithm to $H^{F}(V_0,V_1)$ with input being any set $A$ for which $(A,F) \in \LL_m(s,\lambda)$.
	\sumrise~is simply an asymmetric version of the Kleitman--Winston container algorithm.
	Instead of setting the set of active vertices to be $U = V(H^F)$ at the beginning, we simply set $U = V_1$.
	Again, we stop after placing $s$ vertices in $S$, which are entirely contained in $V_1$.
	It is not hard to check that $A \se U \cup S$ (see Claim~\ref{claim:basic-property-sumrise}), and thus we set $C = U \cup S$ to be the container for $A$, with $U$ being the set of active vertices at the end of the algorithm.

	\subsection*{The container shrinking lemma}
    Let $\C_1$ be the collection of containers given by
    Lemma~\ref{lemma:first-container} with $s$ and $\lambda$ and let $\C_2$ be the collection given by Lemma~\ref{lemma:first-container} by swapping $s$ and $\lambda$.
    If we set $s = \lceil \sqrt{m}\, \rceil$,  $\lambda = \lceil (1+\delta^{-1})\sqrt{m}\, \rceil$ and take
	\begin{align}\label{eq:first-container-C}
		\C' = \left \{ (C_1,C_2): C_1 \in \C_1 \text{ and } C_2 \in \C_2 \right \},
	\end{align}
	we almost obtain Theorem~\ref{thm:simple-containers}, with $\delta = \eps^2$.      
	However, instead of having property (2), we have $|C_1| + |C_2| \le 2(1+\delta)m$ for every $(C_1,C_2) \in \C'$.
	The reduction on $|C_1| + |C_2|$ from roughly $2m$ to $m+\beta$ is done by iteratively applying our so-called container shrinking lemma (see Lemma~\ref{lemma:container-shrinking} below).

	Before we state our container shrinking lemma, let us introduce some notation. For $m,s,\lambda \in \mathbb{N}_{\ge 1}$, $U_0 \se Y+Y$ and $U_1,U_2 \se Y$, we define
	\begin{align}\label{eq:family-F-U}
		\I_{m}(s,\lambda,U_{0},U_{1},U_{2}) \coloneqq 
		\left\{
			(A,B): 
			\begin{aligned}
				& A \se U_{1}
				, B \se U_{2} \text{ and } (A+B)^c \subseteq U_{0}, \\
				& |A+B| \le m \text{ and } |A|\ge s, |B| \ge \lambda.
			\end{aligned}
		\right\},
	\end{align}
	where $(A+B)^c \coloneqq (Y+Y) \setminus (A+B)$.
	When $U_{0} = Y+Y$ and $U_{1} = U_{2} = Y$, or when the sets are clear from the context, we simply omit them from the notation and write $\I_{m}(s,\lambda)$. 
	With this notation, Theorem~\ref{thm:simple-containers} finds a collection of containers for pairs in $\I_{m}(s,\lambda)$, where $s = \lceil \sqrt{m}\, \rceil$ and $\lambda = \lceil (1+\delta^{-1})\sqrt{m}\, \rceil$.

	For sets $U_0 \se Y+Y$ and $U_1, U_2 \se Y$, we define $\HH(U_0,U_1,U_2)$ to be the tripartite 3-uniform hypergraph with parts $U_0, U_1$ and $U_2$, in which vertices $a \in U_1$, $b \in U_2$ and $c \in U_0$ are connected if and only if $c = a+b$.

    To simplify our notation, we use $(A_1,A_2,A_3) \se (B_1,B_2,B_3)$ to indicate that $A_i \se B_i$ for all $i \in [3]$.
\begin{lemma}\label{lemma:container-shrinking}
	Let $\delta >0$ and $m, s, \lambda \in \mathbb{N}_{\ge 1}$ be such that 
	$\lambda s \ge m $.
	Let $U_0 \se Y+Y$ and $U_{1},U_{2} \se Y$, with $|U_1| \ge m/4$.
	There exists a collection $\D$ of size at most $\left ( \binom{n}{s} +1 \right )\binom{n}{\lambda}$ such that the following holds.
	For every pair $(A,B) \in \I_{m}(s,\lambda,U_{0},U_{1},U_{2})$ there exists $(D_0,D_1,D_2) \in \D$ such that
	\[((A+B)^c,A,B) \se (D_0,D_1,D_2) \se (U_0,U_1,U_2).\]
	Moreover, at least one of the following properties are satisfied:
	\begin{enumerate}
		\item The hypergraph $\HH(D_0,D_{1},D_2)$ has strictly less than $\delta|D_{1}||D_2|$ edges;
		\vspace{1mm}
        \item $|D_0| \le |U_{0}|-\frac{\delta m}{6}$ or $|D_{1}| \le \left(1-\frac{\delta}{6}\right)|U_{1}|.$
	\end{enumerate}
\end{lemma}

The main difference between this lemma and a similar result in~\cite{campos2023typical} (see their Theorem 2.1) is that we can obtain a new collection of containers in which $U_0$ or $U_1$ shrinks, while that in their result they could only say that one of the sets $U_0$, $U_1$ or $U_2$ shrinks.
This restriction allows us to iterate Lemma~\ref{lemma:container-shrinking} at most $O(\delta^{-1})$ times to obtain a collection of containers as in Theorem~\ref{thm:simple-containers} (we first shrink $U_1$ as much as we can and later switch the roles of $U_1$ and $U_2$ to shrink $U_2$).
This is part of the reason that allows us to save a log factor (together with having containers of size $\Theta(m)$ in the first step).

	\subsection*{Proof of Theorem~\ref{thm:simple-containers}, assuming Lemmas~\ref{lemma:first-container} and~\ref{lemma:container-shrinking}}\label{sec:main-container}
    We shall need the following supersaturation lemma of Campos, Coulson, Serra, and W\"{o}tzel~\cite{campos2023typical}, which was adapted from Corollary 3.3 in~\cite{campos2020number} for the asymmetric case.
	Roughly speaking, their supersaturation lemma says that if $D_{1},D_{2},W\subseteq{G}$ are finite sets such that $|D_1|+|D_2|$ is large enough compared to $|W|$, then there exist many pairs $(u,v)\in{D_1}\times{D_2}$ summing out of $W$.

\begin{lemma}[~\cite{campos2023typical}, Corollary 3.2]
\label{lemma:supersaturation}
Let $0<\eps <1/4$ and let  $D_{1},D_{2} \se Y$ and $W \se Y+Y$ be sets such that
\begin{align}\label{eq:supersaturation}
	|D_1|+|D_2| \ge (1+2\eps)(|W|+\beta),
\end{align}
where $\beta\coloneqq\beta_{G}\left((1+4\eps)\lvert{W}\rvert\right)$.
Then, there exists at least 
$\eps^2 |D_1||D_2|$ pairs $(u,v) \in D_1 \times D_2$ such that $u+v \in (Y+Y)\setminus W$.
\end{lemma}

Let $0<\delta <1/16$ and set
\[s = \lceil \sqrt{m}\, \rceil \qquad \text{and} \qquad \lambda = \left \lceil (1+\delta^{-1})\sqrt{m} \, \right \rceil.\]
Recall from~\eqref{eq:defn-L-m} that   
\[\LL_m(s,\lambda)=\{(A,B) \in 2^{Y} \times 2^{Y}: |A| \ge s, |B| \ge \lambda \text{ and } |A+B| \le m \}.\]
Now, fix a pair $(A,B) \in \LL_m(s,\lambda)$.
We create a sequence of containers $(D_{0,i},D_{1,i},D_{2,i})_{i\ge 0}$ for $((A+B)^c,A,B)$ as follows.
Let $\D_0$ be the collection of triples $(D_0,D_1,D_2)$ such that $D_0 = Y+Y$ and $(D_1,D_2)\in\C_1\times\C_2$, where $\C_{1}$ is the collection of containers given by Lemma~\ref{lemma:first-container} with $s$ and $\lambda$ and $\C_{2}$ is the collection of containers obtained from Lemma~\ref{lemma:first-container} by swapping $s$ and $\lambda$.
Then, there exists $(D_{0,0},D_{1,0},D_{2,0}) \in \D_0$ such that
\[((A+B)^c,A,B) \se (D_{0,0},D_{1,0},D_{2,0}).\]
Again by Lemma~\ref{lemma:first-container}, we have that the sizes of the sets $D_{1,0}$ and $D_{2,0}$ are at most $(1+\delta)m$.

For the next steps, we shall use Lemma~\ref{lemma:container-shrinking} repeatedly with $s' = \lambda' = \sqrt{m}$.
Our procedure can be formally described as follows.
At the $i$th step, if the process has not yet ended, we will have
constructed a sequence of containers $(D_{0,k},D_{1,k},D_{2,k})_{k=0}^{i-1}$.
For simplicity, we set $\HH_{k}\coloneqq\HH(D_{0,k},D_{1,k}, D_{2,k})$
for every $k \in \{0,\ldots,i-1\}$
and do the following:

\vspace{1mm}
\begin{enumerate}
	\item [S1.] 
	
	\begin{enumerate}
	\item If the first container is large enough, that is, $|D_{1,i-1}| > m/4$,
	then apply Lemma~\ref{lemma:container-shrinking} (with parameters $\delta$ and $s' = \lambda' = \sqrt{m}$) to the pair $(A,B)$ and
	$$U_{0} = D_{0,i-1}, \quad U_{1}
	= D_{1,i-1} \quad \text{ and } \quad U_{2} = D_{2,i-1}$$ to build the next triple of containers.
	By Lemma~\ref{lemma:container-shrinking}, there exists a collection $\D_i$ of size at most $\left(\binom{n}{s'}+1\right)\binom{n}{\lambda'}\le n^{2\sqrt{m}}$ and a triple $(D_{0,i},D_{1,i},D_{2,i}) \in \D_i$ such that 
	$$\hspace*{25mm} ((A+B)^c,A,B) \se (D_{0,i},D_{1,i},D_{2,i}) \se (D_{0,i-1},D_{1,i-1},D_{2,i-1}).$$
	Moreover, we have 
	\[\hspace*{2cm}
	e(\HH_{i}) < \delta |D_{1,i}| |D_{2,i}|
	\quad
	\text{or}
	\quad
	|D_{0,i}| \le |D_{0,i-1}|-\frac{\delta m}{6}
	\quad
	\text{or}
	\quad
	|D_{1,i}| \le \left(1-\frac{\delta}{6}\right)|D_{1,i-1}|.
	\]
	
	\item Set $i \leftarrow i+1$ and repeat the process until 
	$$e(\HH_{i}) < \delta |D_{1,i}| |D_{2,i}|\qquad
	\text{or}
	\qquad 
	|D_{1,i}| \le m/4.$$
	Then, go to stage S2.
\end{enumerate}

\vspace{1mm}
	\item[S2.] If the second container is large enough, that is, $|D_{2,i-1}| > m/4,$ 
	then apply S1.(a)--(c) but swapping the roles of $A$ and $B$ and the roles of $D_{1,i-1}$ and $D_{2,i-1}$.
	Moreover, in order to maintain the symmetry in the construction of our container sequence, we denote by $(D_{0,i},D_{1,i},D_{2,i})$ the triple of containers generated by each step of the process in this case.
	Set $i \leftarrow i+1$ and repeat the process until 
		$$e(\HH_i) < \delta |D_{1,i}| |D_{2,i}|\qquad
		\text{or}
		\qquad 
		|D_{2,i}| \le m/4.$$
		Then, go to stage S3.
	
	\vspace{1mm}
	\item[S3.] If the hypergraph in the previous iteration has few edges \emph{or} the first and second containers are small enough, that is,
	$$e(\HH_{i-1}) < \delta |D_{1,i-1}| |D_{2,i-1}| \qquad \text{or} \qquad \max\{|D_{1,i-1}|,|D_{2,i-1}|\} \le m/4,$$ 
	then set 
	$k^{\ast} = i-1.$
	The process terminates and its output is set to be the sequence $(D_{0,i},D_{1,i},D_{2,i})_{i=0}^{k^{\ast}}$.
\end{enumerate}

		\begin{claim}\label{claim:container-sequence}
			For every $(A,B) \in \LL_m(s,\lambda)$ the process above terminates in at most $k^{\ast} \le 2^{5}\delta^{-1}$ steps.
		\end{claim}

\begin{proof}[Proof of Claim~\ref{claim:container-sequence}]
	Let $j \in \{1,2\}$, $i\ge1$ and suppose that in its $i$th step the process is in stage S$j$.
	Then,
	we have $e(\HH_i) < \delta |D_{1,i}| |D_{2,i}|$ (and in this case the process finishes at time $i$),
	or at least one of the sets $D_{0,i-1}$ and $D_{j,i-1}$ has been shrunk:
	\[
		\hspace*{2cm}
		|D_{0,i}| \le |D_{0,i-1}|-\frac{\delta m}{6}
		\quad
		\text{or}
		\quad
		|D_{j,i}| \le \left(1-\frac{\delta}{6}\right)|D_{j,i-1}|.
	\]
	As $(A+B)^c \se D_{0,i} \se Y+Y$ and $|(A+B)^c| \ge |Y+Y|-m$, it follows that we cannot have more than $6\delta^{-1}$ values of $i$ for which $D_{0,i}$ is shrunk in stage S$j$.
	Similarly, since $|D_{j,0}|\le 2m$ and stage S$j$ lasts until $|D_{j,i}| \le m/4$,
	it follows that $D_{j,i}$ is shrunk at most $6\delta^{-1}\ln 8$ times. 
	Therefore, the process terminates in at most $k^{\ast} \le 6\delta^{-1}(1+2\ln 8) \le 2^5 \delta^{-1}$ steps.
\end{proof}

Define $\mathcal{C}$ as the collection of all triples $(D_{0,k^{\ast}},D_{1,k^{\ast}},D_{2,k^{\ast}})$ that are the final elements of sequences generated by the aforementioned procedure.
By Claim~\ref{claim:container-sequence}, the length of each sequence $(D_{0,i}, D_{1,i}, D_{2,i})_i$ produced by this process is at most $2^5\delta^{-1}$. Furthermore, since $\D_0$ has size at most $\binom{n}{s}^2\binom{n}{\lambda}^2$ and since each triple $(D_{0,i}, D_{1,i}, D_{2,i})$ is derived from a collection $\mathcal{D}_i$ of size at most $n^{2\sqrt{m}}$ for all $i \ge 1$, it follows that 
$$|\mathcal{C}| \le \binom{n}{s}^2\binom{n}{\lambda}^{2}n^{2\sqrt{m}\cdot2^{5}\delta^{-1}} \le n^{4\sqrt{m}+(2^{6}+2)\delta^{-1}\sqrt{m}} \le n^{2^7\delta^{-1}\sqrt{m}}.$$

By construction, we have that for every $(A,B) \in \LL_m(s,\lambda)$ there exists $(C_0,C_{1},C_2) \in \C$ such that $((A+B)^c,A,B) \se (C_0,C_{1},C_2)$. 
Moreover, by the termination rule of the process, we have
$$ |C_{1}|,|C_2| \le \dfrac{m}{4} 
\qquad
\text{or}
\qquad
e(\HH) < \delta|C_{1}||C_2|,$$
where $\HH = \HH(C_0,C_{1},C_2)$.
By Lemma~\ref{lemma:supersaturation},
if $e(\HH) < \delta |C_{1}| |C_{2}|$, then 
\[|C_{1}|+|C_{2}| \le (1+2\sqrt{\delta})(m+\beta),
\]
where $\beta = \beta_G\big((1+4\sqrt{\delta})m\big)$.
We complete our proof by setting $\delta = \eps^2$.\qed

\section{sumrise: the first container algorithm}\label{sec:first-container}
In this section we introduce our first container algorithm, which we call~\sumrise.
It is used to derive Lemma~\ref{lemma:first-container} and also an auxiliary lemma which shall be used in Section~\ref{sec:second-container} to prove Lemma~\ref{lemma:container-shrinking}.
Throughout this section, 
we let $n$ be a positive integer, and we fix an abelian group $G$, an $n$-element subset $Y$ of $G$ and subsets $U_0 \se V_0 \coloneqq Y+Y$ and $U_1 \se V_1 \coloneqq Y$.

For $s, m, \lambda \in \mathbb{N}_{\ge 1}$, define
\[\K_m(s,\lambda) \coloneqq \left \{ (A,F) \in \LL_m(s,\lambda) : 
	A \se U_1 \text{ and } (A+F)^c \se U_0\right \}.\]
For simplicity, we omit $U_0$ and $U_1$ from our notation.
Note that when $U_0 = V_0$ and $U_1 = V_1$ we have $\K_m(s,\lambda) = \LL_m(s,\lambda)$.

The algorithm~\sumrise~shall be applied to pairs belonging to the set $\K_m(s,\lambda)$. Initially, it may seem peculiar why we are interested in collections of sets other than $\mathcal{L}_m(s,\lambda)$. The benefit of expressing everything in terms of general sets like $\K_m(s,\lambda)$ will become evident in Section~\ref{sec:second-container}, where we introduce our second container algorithm and leverage it to prove Lemma~\ref{lemma:container-shrinking}.

As an input,~\sumrise~receives sets $A,F \se Y$ and a natural number $s$.
As an output,~\sumrise~returns subsets $C_0 \se U_0$ and $S, C_1 \se U_1$ such that $S \se A \se S \cup C_1$ and $(A+F)^c \se C_0$.
We call $S$ the \emph{fingerprint} of $A$ and we call $(C_0,S \cup C_1)$ the pair of \emph{containers} for $(A+F)^c$ and $A$, respectively.
Three sequences are employed to construct these sets: $(S_i)_{i=0}^{\tau}$, $(U_{0,i})_{i=0}^{\tau}$ and $(U_{1,i})_{i=0}^{\tau}$, where $\tau$ denotes the number of iterations of the while loop of~\sumrise. 
Their first elements are set to be $S_0=\emptyset$, $U_{0,0} = U_{0}$ and $U_{1,0} = U_{1}$ (see line~\ref{sumrise:initialize}). 
To construct the next elements, \sumrise~makes use of an auxiliary function $\sigma: \{H \se H^F(V_0,V_1): V(H)\cap V_1 \neq \emptyset\} \to V_1$ with the following property. For any $H$ in the domain of $\sigma$, $\sigma(H)$ is a vertex of maximum degree in $H$ among the vertices in $V(H)\cap V_1$.
This function is fixed before the algorithm is implemented.

At the $i$th step, \sumrise~sets $H = H^F(U_{0,i},U_{1,i})$, $u = \sigma(H)$ and asks if $u$ is in $A$. 
If $u$ is in $A$, then it sets  $U_{0,i+1}=U_{0,i} \setminus N(u)$, $U_{1,i+1}=U_{1,i}\setminus \{u\}$ and $S_{i+1} = S_i \cup \{u\}$ (see lines~\ref{sumrise:update-S-1} and~\ref{sumrise:update-U-1}).
Here, $N(u)$ denotes the neighbourhood of the vertex $u$ in the graph $H^F(V_0,V_1)$, that is, $N(u) \coloneqq u+F$.
If $u$ is not in $A$, then it sets $U_{0,i+1}=U_{0,i}$, $U_{1,i+1}=U_{1,i}\setminus \{u\}$ and $S_{i+1} = S_i$ (see lines~\ref{sumrise:update-S-2} and~\ref{sumrise:update-U-2}).

One can check that $S_i \se A \se S_i \cup U_{1,i}$ in all steps.
The algorithm terminates its while loop when $|S_i| = s$ (see line~\ref{sumrise:while-loop-start}).
Upon termination, it sets $S = S_i$, $C_0 = U_{0,i}$ and $C_1 = U_{1,i}$ (see line~\ref{sumrise:set-S-C}).
The~\sumrise~algorithm is formally described below.

\vspace{6pt}

\begin{algorithm}[H]
	\DontPrintSemicolon

	\vspace{3mm} 

	\KwInput{$A,F \se Y$ and $s \in \mathbb{N}_{\ge 1}$}

	\vspace{1mm} 

	\KwOutput{
		sets $S \se A$, $C_0 \se U_{0}$ and $C_{1} \se U_{1}$}

	\vspace{3mm} 

	\tcc{Initialize:}

	\nl $i=0$, 
	$S_i= \emptyset$, $U_{0,i} = U_{0}$ and $U_{1,i} = U_{1}$
	{\label{sumrise:initialize}}\;

	\vspace{1mm} 

	\nl \While{
	$|S_i| < s$
	\label{sumrise:while-loop-start} 
	}			
	{ \vspace{1mm} 
	\nl $H \gets H^F(U_{0,i},U_{1,i})$ and $u \gets \sigma(H)$\;

	\vspace{1mm} 

	\nl \If{
		\vspace{1mm} 
		$u \in A$
        \label{sumrise:u-in-A}
	}							
	{
	\nl set $S_{i+1} = S_i \cup \{u\}$ \label{sumrise:update-S-1}
		\vspace{1mm} 
	
		\nl set $U_{1,i+1} = U_{1,i}\setminus \{u\}$ and $U_{0,i+1} = U_{0,i} \setminus N(u)$ 
		\label{sumrise:update-U-1}\; 
	}
	\Else{
		\label{sumrise:else}
		\nl set  $S_{i+1}=S_i$ \label{sumrise:update-S-2}
		\vspace{1mm} 
		
		\nl set $U_{1,i+1} = U_{1,i}\setminus \{u\}$ and $U_{0,i+1} = U_{0,i}$ 
		\label{sumrise:update-U-2}\;
	}
	\nl $i \mapsto i+1$ {\label{sumrise:increase}}\;}  \vspace{3pt}

	\nl \label{sumrise:set-S-C} \textbf{set}
	$S = S_i$, $C_{0} = U_{0,i}$ and $C_1 = U_{1,i}$  
	\vspace{3pt}

\Return{
	$S$, $C_{0}$ and $C_{1}$.
}{\label{sumrise:return}}

\caption{\label{sumrise} \sumrise}
\end{algorithm}

\vspace{3mm}

Our next claim establishes the most basic properties about our algorithm.
For each set $A$ of size at least $s$ and each set $F$, let \sumrise$(s,A,F)$ denote the \sumrise~algorithm with input $(A,F)$ and stop condition $|S_i| = s$ (see line~\ref{sumrise:while-loop-start}).

\begin{claim}\label{claim:basic-property-sumrise}

	Let $s, \lambda, m \in \mathbb{N}_{\ge 1}$.
	If $(A,F) \in \K_m(s,\lambda)$, then \textnormal{\textsc{sumrise}}$(s,A,F)$ is executed without error and the objects $S$ and $(C_{0},C_{1})$ output from it satisfy the following properties:
	\begin{enumerate}
		\item [$(1)$] $S \se A \se S \cup C_1 \se U_1$;
        \vspace{1mm}
		\item [$(2)$] $(A+F)^c \se C_{0} \se U_0$;
        \vspace{1mm}
		\item [$(3)$]$|S| = s$ and $S \cap C_1 = \emptyset$.
	\end{enumerate}

\end{claim}

\begin{proof}

	Let $i \in \mathbb{N}_{\ge 0}$ and
	suppose that \sumrise$(s,A,F)$ executes the $i$th iteration of its while loop.
	For simplicity, denote $u_i = \sigma\big(H^F(U_{0,i},U_{1,i})\big)$.
	In the execution of both if-conditional and else-conditional, 
	the algorithm sets $U_{1,i+1} = U_{1,i}\setminus \{u_{i}\}$, see lines~\ref{sumrise:update-U-1} and~\ref{sumrise:update-U-2}.
	Since $|A| \ge s$ and $A \se U_{1,0} = U_1$, and since $S_{i+1} = S_{i} \cup \{u_{i}\}$ whenever $u_i \in A$, 
	from these observations it follows that the while loop terminates after a finite number of iterations.
	If we let $\tau$ denote the stopping time of the while loop, then from line~\ref{sumrise:set-S-C} we obtain $S = S_{\tau}$, and hence $|S| = s$.

	Now, let us prove by induction that 
	\begin{align}\label{eq:induction-claim}
		(A+F)^c \se U_{0,i} \se U_{0},
		\qquad
		S_i \cap U_{1,i} = \emptyset 
		\qquad
		\text{and}
		\qquad
		S_i \, \se \, A \, \se \, S_i \cup U_{1,i} \se U_{1}
	\end{align}
	for every $i \in \{0,\ldots, \tau\}$.
	For the base case $i = 0$, our assertion follows from the fact that $A \in \K_m(s,\lambda)$ and and the initial conditions $S_0 = \emptyset$, $U_{0,0} = U_{0}$ and $U_{1,0} = U_{1}$.
	Now, suppose it is true for some $i$ and let us prove it for $i+1$. We have two cases:
	\begin{enumerate}
		\item[$*$] If $u_{i} \notin A$, then by lines~\ref{sumrise:update-S-2} and~\ref{sumrise:update-U-2} we have 
		$S_{i+1} = S_i$, 
		$U_{0,i+1} = U_{0,i}$ and 
		$U_{1,i+1} = U_{1,i} \setminus \{u_i\}$. 
		This together with the induction hypothesis shows that~\eqref{eq:induction-claim} is true with $i$ replaced by $i+1$.
        \vspace{1mm}
		\item[$*$] If $u_i \in A$, then lines~\ref{sumrise:update-S-1} and~\ref{sumrise:update-U-1} together with the induction hypothesis immediately imply that 
		$S_{i+1} \se A$, that $S_{i+1} \cap  U_{1,i+1} = \emptyset$ and that $A \se S_{i+1} \cup U_{1,i+1} \se U_1$.
		Now, note that since $u_i \in A$, we have $N(u_i) = u_i + F \se A+F$, and hence $(A+F)^c \se N(u_i)^c$.
		Since $U_{0,i+1} = U_{0,i} \setminus N(u_i)$ (see line~\ref{sumrise:update-U-1}) and $(A+F)^c \se U_{0,i} \se U_0$ by the induction hypothesis, it follows that $(A+F)^c \se U_{0,i+1} \se U_0$.

	\end{enumerate}

	Finally, it follows from line~\ref{sumrise:set-S-C} that $C_0 = U_{0,\tau}$ and $C_1 = U_{1,\tau}$.
	This together with~\eqref{eq:induction-claim} completes our proof.
\end{proof}

	Our next claim establishes another basic property of the \sumrise~algorithm.
	Roughly speaking, it says that we can recover the entire output once we know the final fingerprint $S$.

\begin{claim}\label{claim:container-equal} % {claim:container-equal}
	Let $s, \lambda, m \in \mathbb{N}_{\ge 1}$
	and suppose that $(A,F)$ and $(A',F)$ belong to $\K_m(s,\lambda)$.
	If \textnormal{\sumrise}$(s,A,F)$ and \textnormal{\sumrise}$(s,A',F)$ output the same fingerprint, then they also output the same pair $(C_0,C_1)$.

\end{claim}

	\begin{proof}
		 Let $\tau$ and $\tau'$ be the stopping times of the while loops of \sumrise $(s,A,F)$ and \sumrise $(s,A',F)$, respectively.
		It suffices to show that for every $i \in \{1,\ldots,\tau\}$, the $i$th iterations of the while loops in \sumrise $(s,A,F)$ and \sumrise $(s,A',F)$ are executed in exactly the same way.
		Suppose for contradiction that this is not the case and let $i$ be the first loop iteration for which \sumrise $(s,A,F)$ and \sumrise $(s,A',F)$ differ.
		As all iterations before $i$ coincide in \sumrise $(s,A,F)$ and \sumrise $(s,A',F)$, we have
		\[
			U_{0,i-1}(A) = U_{0,i-1}(A'), \quad
			U_{1,i-1}(A) = U_{1,i-1}(A') \quad
			\text{and}
			\quad 
			S_{i-1}(A)=S_{i-1}(A').
		\]

		Without loss of generality, suppose that at the $i$th iteration, \sumrise $(s,A,F)$ executes the if-conditional and \sumrise $(s,A',F)$ executes the else-conditional.
		Then, we have $S_i (A)= S_{i-1}(A) \cup \{u\}$ and $S_i (A')= S_{i-1}(A')$, where $u \coloneqq \sigma(H^F(U_{0,i-1},U_{1,i-1}))$.
		Now, we note that the sequence of fingerprints is always a chain. That is,
		$S_0(A) \se S_1(A)  \se \cdots \se S_{\tau}(A) $, and hence $u \in  S_{\tau}(A)$.
		As $S_{\tau}(A) = S_{\tau'}(A')$ and $S_{\tau'}(A') \se A'$ (see Claim~\ref{claim:basic-property-sumrise}), it follows that $u \in A'$.
		This is a contradiction, as \sumrise$(s,A',F)$ executed the else-conditional because $u \notin A'$.
	\end{proof}

	Our next lemma bounds the value of $|S| + |C_1|$, given that the sets $S$ and $(C_0, C_1)$ are produced as output by the \sumrise~algorithm.
	In its proof, it is useful to have an upper bound on the maximum degree of $H^F(V_0,V_1)$.
	Note that for each $u \in F$, the set of edges $\{\{v,v+u\}: v \in V_1\} \se H^F(V_0,V_1)$ forms a matching, and hence 
	\begin{align}\label{eq:max-degree-H}
		d(v)\le |F| \text{ for every } v \in V_0 \cup V_1.
	\end{align}

	\begin{lemma}\label{lemma:sumrise-fingerprint}
		Let $\delta >0$ and $m, \lambda,s \in \mathbb{N}_{\ge 1}$ be such that $\delta \lambda s \ge (1+\delta)m$.
		If $(A,F) \in \K_m(s,\lambda)$, then
		the sets $S$ and $C_1$ output by \textnormal{\textsc{sumrise}}$(s,A,F)$ satisfy
		\[|S|+|C_1| \le (1+\delta)m.\]
	\end{lemma}

	\begin{proof}

		Let $\tau$ be the stopping time of the while loop of \sumrise$(s,A,F)$.
		It follows from lines~\ref{sumrise:update-U-1},~\ref{sumrise:update-U-2} and~\ref{sumrise:set-S-C} that
		\[ 
			C_0 = U_{0,\tau} = U_{0} - \bigcup_{u \in S} N(u),
		\]
		where $N(u) = u+F$.
		As the set $U_0 \setminus C_0$ of removed vertices is contained in $A+F$ and $|A+F| \le m$, we have $|U_0 \setminus C_0| \le m$.

		In $H^F(V_0,V_1)$, note that each vertex in $V_1 = Y$ has degree $|F|$, and hence 
		$e\big(H^F(V_0,V_1)\big) = |F||Y|$.
		As the maximum degree in $H^F(V_0,V_1)$ is $|F|$ (see~\eqref{eq:max-degree-H}), 
		it follows that the number of edges in the graph $H^F(C_{0}, S \cup C_1)$ is at least
		\begin{align}\label{eq:lower-bound-eH}
			\nonumber
			e\big ( H^F(C_{0}, S \cup C_1) \big) & \ge 
			|F||Y|
			-\big ( |V_0  \setminus C_0  | 
			+ |V_1 \setminus (S \cup C_1) | \big )\cdot |F|\\
			& \ge |F| |Y|- \big ( m+|Y|-|C_1|-|S| \big) \cdot |F| \nonumber\\
			& = \big (|C_1| + |S| -m\big) \cdot |F|.
		\end{align}
        Moreover, as $S$ is an isolated set in $H^F(C_0,S \cup C_1)$,
        we observe that $H^F(C_0,C_1)$ and $H^F(C_0,S \cup C_1)$ have the same number of edges.

	Suppose for contradiction that $|S|+|C_1| > (1+\delta)m$.
	For simplicity, let $d$ be the average degree of a vertex of $H^F(C_0,C_1)$ in $C_1$.
	As $S$ and $C_1$ are disjoint (see Claim~\ref{claim:basic-property-sumrise}--(3)),
	it follows from~\eqref{eq:lower-bound-eH} that
	\begin{align}\label{eq:avg-degree-H}
		d=\dfrac{e\left(H^F(C_0,C_1)\right)}{|C_1|} \ge \dfrac{e\big ( H^F(C_{0}, S \cup C_1) \big)}{|S \cup C_1|} \ge \left(1-\dfrac{m}{|S|+|C_1| }\right)|F| > \dfrac{\delta}{1+\delta}|F| \ge \dfrac{m}{s}.
	\end{align}
	Note that in the last inequality we used the assumptions that $\delta\lambda s \ge (1+\delta)m$ and $\lvert{F}\rvert\ge \lambda$. Every time \sumrise~executes line~\ref{sumrise:update-U-1}, it removes the neighbourhood of a vertex which is being added to the fingerprint.
	By the maximum degree order given by $\sigma$, we have that these neighbourhoods have size at least $d$.
	By~\eqref{eq:avg-degree-H}, it follows that \sumrise~removes strictly more than $s \cdot m/s = m$ vertices from $U_0$.
	This is a contradiction, since  the only vertices removed from $U_0$ are those in $U_0 \setminus C_0$ and $|U_0 \setminus C_0|\le m$.
	\end{proof}

	Our next lemma states that if $|U_{1}| \ge m/4$, $(A,F) \in \K_m(s,\lambda)$ with $s\lambda \ge m$, and $H^F(U_{0},U_{1})$ is sufficiently dense, then \sumrise$(s,A,F)$ ``shrinks'' one of the sets $U_{0}$ and $U_{1}$.
	This lemma shall be used in Section~\ref{sec:second-container} to derive Lemma~\ref{lemma:container-shrinking}.

	\begin{lemma}\label{lemma:graph-shrinking}
		Let $\delta>0$ and $m, \lambda, s \in \mathbb{N}_{\ge 1}$ be such that $\lambda s \ge m$ and suppose that $|U_1| \ge m/4$.
		If $(A,F) \in \K_m(s,\lambda)$ and $e \big (H^F(U_{0},U_{1})\big)\ge\delta|F||U_{1}|$, then the sets $S$ and $(C_0,C_1)$ output by \textnormal{\textsc{sumrise}}$(s,A,F)$ satisfy
		\begin{align*}
			|C_{0}| \le |U_{0}|-\frac{\delta m}{6} \qquad \text{or} \qquad |S|+|C_1| \le \left ( 1-\frac{\delta}{6}\right )|U_{1}|.
		\end{align*}
	\end{lemma}

 \begin{proof}
 	Suppose for contradiction that $|C_{0}| > |U_{0}|-\frac{\delta m}{6}$ and $|S|+|C_1| > \big(1-\frac{\delta}{6}\big)|U_{1}|$ hold simultaneously.
	Recall that $|S| = s$ (see Claim~\ref{claim:basic-property-sumrise}--(3)).
	Every time a vertex is added into $S$, its neighbourhood is removed from $U_0$ (see line~\ref{sumrise:update-U-1}). Then, the neighbourhood of $S$ in $H^F(U_{0},U_{1})$ is exactly equal to $U_0 \setminus C_0$.
	In particular, we have that $H^{F}(C_0,C_1)$ and $H^{F}(C_0,C_1 \cup S)$ have the same number of edges.
 	As $e \big (H^F(U_{0},U_{1})\big)\ge\delta|F||U_{1}|$ and its maximum degree is at most $|F|$ (see~\eqref{eq:max-degree-H}), it follows that
 	\begin{align}\label{eq:edges-H-eps}
 		e\big(H^F(C_{0},C_{1})\big) & = e\big ( H^{F}(C_0,C_1 \cup S) \big) \nonumber\\
		& \ge \delta |F||U_{1}| - (|U_{0} \setminus C_{0}| + |U_{1}
		 \setminus (C_{1}\cup S)|)|F| \nonumber\\
		& > \delta |F||U_{1}
		| - \left( \dfrac{\delta m}{6}+\dfrac{\delta |U_{1}
		|}{6}\right)|F|\nonumber\\
		& \ge \frac{\delta}{6}\lambda|U_{1}|. 
 	\end{align}
	In the last inequality, we used that $|U_1| \ge m/4$.
 	  Let $d$ be the average degree of a vertex of $H^F(C_{0},C_{1})$ in $C_1$.
 	 As $|U_1| \ge |C_1|$, it follows from~\eqref{eq:edges-H-eps} that 
 	 \begin{align}\label{eq:avg-degree-H-eps}
 	 	d = \dfrac{e\big(H^F(C_{0},C_{1})\big)}{|C_{1}|} > \dfrac{\delta \lambda}{6}.
 	 \end{align}
 	 
 	 Recall that every time \sumrise~executes line~\ref{sumrise:update-U-1} it removes the neighbourhood of a vertex which is being added to the fingerprint.
 	 By the maximum degree order given by $\sigma$, we have that these neighbourhoods have size at least $d$.
 	 By~\eqref{eq:avg-degree-H-eps}, it follows that \sumrise~removes strictly more than $s \cdot \delta \lambda/6 \ge \delta m/6$ vertices from $U_{0}$.
 	 This leads to a contradiction, since  we assumed that $|C_{0}| > |U_{0}|-\frac{\delta m}{6}$.
 \end{proof}

\subsection*{Proof of Lemma~\ref{lemma:first-container}}\label{Sec:proof-first-container}

For a set $F \in \binom{Y}{\ge \lambda}$, define
\[ \C_F \coloneqq \Big\{ S(A) \cup C_1(A): (A,F) \in \K_m(s,\lambda)\Big\}.
\]
By Claim~\ref{claim:container-equal}, we have $C_1(A) = C_1(A')$ whenever $S(A) = S(A')$. Moreover, since $|S(A)| = s$ for every $\K_m(s,\lambda)$ (by Claim~\ref{claim:basic-property-sumrise}) and $|S(A)| + |C_1(A)| \le (1+\delta)m$ (by Lemma~\ref{lemma:sumrise-fingerprint}), it follows that 
\begin{align}\label{eq:container-size-lemma}
	|\C_F| \le \binom{n}{s} \quad \text{and} \quad |C| \le (1+\delta)m \text{ for every } C \in \C_F.
\end{align}
Now, set our collection $\C$ to be
\[ \C = \bigcup_{F \in \binom{Y}{\lambda}} \C_F.\]
Items (1) and (2) follow from~\eqref{eq:container-size-lemma}. Item (3) follows from Claim~\ref{claim:basic-property-sumrise} and the fact that $\K_m(s,\lambda) = \LL_m(s,\lambda)$ when $U_0 = V_0$ and $U_1 = V_1$.\qed

\section{Sunset: the second container algorithm}\label{sec:second-container}

In this section, we prove Lemma~\ref{lemma:container-shrinking} and introduce our second container algorithm, called \sunset, which is the main ingredient behind its proof.
Similar to Section~\ref{sec:first-container}, 
throughout this section we let $n$ be a positive integer, and we fix an abelian group $G$, an $n$-element subset $Y$ of $G$, and
we fix sets $U_{0},U_{1},U_{2}$ such that $U_{0} \se Y+Y$ and $U_{1}, U_{2} \se Y$.

As in Section~\ref{sec:first-container}, we find it instructive to first provide an informal overview of our algorithm.
As an input, \sunset~receives a pair $(A,B)$ such that $A \se U_1$, $B \se U_2$ and $(A+B)^c \se U_0$, 
a real number $\delta$ and natural numbers $s$ and $\lambda$.
As an output, \sunset~returns subsets $(S,F)\se (A,B)$ and $(C_{0},C_{1},C_{2}) \se (U_{0},U_{1},U_{2})$
such that 
\[
	S \se A \se S \cup C_{1}, 
	\quad 
	F \se B \se F \cup C_{2}, 
	\quad 
	\text{and}
	\quad
	(A+B)^c \se C_{0}.
\]

As in the previous section, we call $(S,F)$ the \emph{fingerprint} pair for $(A,B)$ and we call $(C_{0},C_{1}\cup S,C_{2} \cup F)$ the \emph{container triple} for $((A+B)^c,A,B)$, respectively.

In its while loop from line~\ref{sunset:while-loop-start} to~\ref{sunset:increase},~\sunset~builds two sequences of sets, namely $(F_i)_{i=0}^{\tau}$ and $(U_{2,i})_{i=0}^{\tau}$, where $\tau$ denotes the stopping time of the while loop.
Their first elements are set to be $F_0 = \emptyset$ and $U_{2,0} = U_{2}$, respectively (see line~\ref{sunset:initialize}).
As in Section~\ref{sec:first-container}, in order to build such sequences,
\sunset~makes use of an auxiliary function $\sigma: \{\HH \se \HH(V_0,V_1,V_2): V(\HH)\cap V_2 \neq \emptyset\} \to V_2$ with the following property. For every $\HH$ in the domain of $\sigma$, $\sigma(\HH)$ is a vertex of maximum degree in $\HH$ among the vertices in $V(\HH)\cap V_2$.
This function is fixed before the algorithm is implemented.

The~\sunset~algorithm builds $(F_i)_{i=0}^{\tau}$ and $(U_{2,i})_{i=0}^{\tau}$ by going vertex by vertex in $U_{2}$ according to the ordering $\sigma$ (see line~\ref{sunset:sigma})
and asking whether the selected vertex, say $u$, is in $B$ or not. If it is, then it adds $u$ to the fingerprint (see line~\ref{sunset:add-to-F}); otherwise, it does not (see line~\ref{sunset:not-add-to-F}). 
Then, \sunset~updates $U_{2}$ by removing $u$ from it.
The algorithms ends its while loop after adding $\lambda$ vertices to its fingerprint.
It is not important to know how much $U_{2}$ has shrunk during the while loop.

Once~\sunset~has built 
$(F_i)_{i=0}^{\tau}$ and $(U_{2,i})_{i=0}^{\tau}$, it sets $F = F_{\tau}$ and $C_2 = U_{2,\tau}$ (see line~\ref{sunset:define-2}).
We observe that until line~\ref{sunset:define-2} 
the sets $U_{0}$ and $U_{1}
$ have not yet been updated.
In line~\ref{sunset:if}, \sunset~enters an if-conditional and asks whether $e(\HH_{\tau}) \ge \delta|U_{1}
||C_2|$ or $e(H^F(U_0,U_1)) \ge \delta|F||U_1|$.
If none of these conditions are satisfied, then~\sunset~stops and output the following sets: $(S,F) = (\emptyset, F_{\tau})$ and $(C_{0}, C_{1},C_{2}) = (U_{0}, U_{1}, U_{2,\tau})$ (see line~\ref{sunset:else}).
Otherwise,~\sunset~runs~\sumrise$(s,A,F)$ in line~\ref{sunset:run-sumrise}, which is our graph container algorithm applied to $H^{F}(U_{0},U_{1}
)$.
We will see that the condition $e(\HH_{\tau}) \ge \delta|U_{1}||C_2|$ also ensures that $H^{F}(U_{0},U_{1}
)$ has at least $\delta |F||U_1|$ edges, and hence~\sumrise~is able to shrink either $U_{0}$ or $U_{1}
$.
We then let $S$ and $(C_0,C_{1})$  be the objects output by~\sumrise$(s,A,F)$ (see line~\ref{sunset:define-3}).
Finally,~\sunset~output 
$(S,F)$ and $(C_{0},C_{1},C_{2})$.
The algorithm~\sunset~is formally described below.

\begin{algorithm}[H]
	\DontPrintSemicolon
	
	\vspace{3mm} 
	
	\KwInput{$\delta \in \mathbb{R}_{>0}$, $s,\lambda \in \mathbb{N}_{\ge 1}$ and $(A,B)$ such that $A \se U_1$, $B \se U_2$ and $(A+B)^c \se U_0$}
	
	\vspace{1mm} 
	
	\KwOutput{
		A pair $(S,F)$ and a triple $(C_0,C_1,C_2)$ such that 
	
	 \hspace{53pt} $S \se A \se S \cup C_{1}$, $F \se B \se F \cup C_{2}$ and $(A+B)^c \se C_{0}$.
			}

	\vspace{3mm} 
	
	\tcc{Initialize:}

	\nl $i=0$, $F_i = \emptyset$ and $U_{2,i}=U_{2}$.
	{\label{sunset:initialize}}\;
	
	\vspace{1mm} 
	
	\nl \While
	{
		\label{sunset:while-loop-start}
		$|F_i| < \lambda$
	}			
	{ 
		\vspace{1mm} 
		\nl $\HH_i \gets \HH(U_{0},U_{1},U_{2,i})$ 
		and 
		$u \gets \sigma(\HH_{i})$ 
		\label{sunset:sigma}\;
		\vspace{1mm}

		\nl \If{
			\vspace{1mm} 
			$u \in B$
			\label{sunset:u-in-B}
		}							
		{
			\nl set $F_{i+1} = F_i \cup \{u\}$ and $U_{2,i+1} = U_{2,i} \setminus \big \{u\}$ 
			\label{sunset:add-to-F}\; 
		}

		\nl \Else{
			\label{sunset:u-not-in-B}
			\nl 
			set  $F_{i+1}=F_i$ and $U_{2,i+1} = U_{2,i} \setminus \{u\}$
			{\label{sunset:not-add-to-F}}
			\;
			}

		\nl $i \mapsto i+1$ {\label{sunset:increase}}\;}
	\vspace{1mm} 
	\nl set $F = F_{i}$, $C_2 = U_{2,i}$ and $H = H^F(U_0,U_1)$ \label{sunset:define-2}\;
	\vspace{1mm}

	\nl \If{\vspace{1mm}  
	$e(\HH_{i}) \ge \delta |U_{1}||C_{2}|$ or $e(H) \ge \delta |F||U_1|$ 
	\label{sunset:if}}
	{
	\vspace{1mm} 
	\nl run \sumrise$(s,A,F)$ let $S$ and $(C_{0},C_{1})$ be objects in its output \label{sunset:run-sumrise}\label{sunset:define-3}
	
	} 

	\nl \Else{ \nl set $S = \emptyset$, $C_{0} = U_{0}
	$ and $C_{1}  = U_{1}$ \label{sunset:else}}

	\nl \Return{
		$(S,F)$ and $(C_{0},C_{1},C_{2})$. 
	}{\label{sunset:return}}
	
	\caption{\label{sunset} \sunset}
\end{algorithm}

	\vspace{3mm}

Recall the definition of $\I_{m}(s,\lambda,U_{0},U_{1},U_{2})$ in~\eqref{eq:family-F-U}.
Since $U_0, U_1$ and $U_2$ are fixed in this section,
for simplicity, we denote
$$\I_m(s,\lambda)\coloneqq \I_{m}(s,\lambda, U_{0},U_{1},U_{2}).$$

Similarly to Claim~\ref{claim:basic-property-sumrise}, our next claim establishes the most basic property about the \sunset~algorithm.
We write $\sunset_{\delta}(s,\lambda,A,B)$ to refer to our algorithm when $(A,B)$, $\delta, s$ and $\lambda$ are in the input.

\begin{claim}\label{claim:2nd-container-fingerprint}

Let $\delta>0$ and $s, \lambda, m \in \mathbb{N}_{\ge 1}$.
If $(A,B) \in \I_m(s,\lambda)$, then $\sunset_{\delta}(s,\lambda,A,B)$ is executed without error and the objects $(S,F)$ and $(C_{0},C_{1},C_{2})$ output from it satisfy the following properties:
\begin{enumerate}
	\item [$(1)$] $(A+B)^c \se C_{0} \se U_0$;
	\vspace{1mm}
	\item [$(2)$] $|S| \in \{0,s\}$, $S \cap C_{1} = \emptyset$ and $S \, \se \, A \, \se \, S \cup C_{1} \se U_1$;
	\vspace{1mm}
	\item [$(3)$]$|F| = \lambda$, $F \cap C_2 = \emptyset$ and $F \, \se \, B \, \se \, F \cup C_2 \se U_2$.
\end{enumerate}

\end{claim}

\begin{proof}
We first note that, since $|B| \ge \lambda$, 
the if-conditional in line~\ref{sunset:u-in-B} is executed exactly $\lambda$ times.
This implies that no error occurs during the execution of the while loop in lines~\ref{sunset:while-loop-start} to~\ref{sunset:increase},
which then finishes in finite time.
Let $(F_i)_{i=0}^{\tau}$ and $(U_{2,i})_{i=0}^{\tau}$ be the sequences built in the while loop, where $\tau$ is its stopping time.
By line~\ref{sunset:define-2} we have $F = F_{\tau}$, and hence $|F| = |F_{\tau}|=\lambda$.

Now we prove by induction that
\begin{align}\label{eq:basic-property-sunset}
	F_i \cap U_{2,i} = \emptyset
	\qquad
	\text{and}
	\qquad
	F_i \, \se \, B \, \se \, F_i \cup U_{2,i} \se U_2
\end{align}
for all $i \in \{0,\ldots,\tau\}$.
As $F_0 = \emptyset$ and $U_{2,0} = U_{2}$, our assertion is true for $i=0$. 
Now, suppose it is true for some $i$ and let us prove it for $i+1$.
For simplicity, we set $u = \sigma(\HH_i) \in U_2$.
We have two cases:

\begin{enumerate}
	\item [$\ast$] If $u \notin B$, then by line~\ref{sunset:not-add-to-F} we have 
	$F_{i+1} = F_i$ and $U_{2,i+1} = U_{2,i} \setminus \{u\}$. 
	This together with the induction hypothesis shows that 
	$F_{i+1} \cap U_{2,i+1} = \emptyset$ and that $F_{i+1} \, \se \, B \, \se \, F_{i+1} \cup U_{2,i+1} \se U_2$. 
\vspace{1mm}
	\item [$\ast$] 
	If  $u \in B$, then line~\ref{sunset:add-to-F} together with the induction hypothesis shows that 
	$F_{i+1} \cap U_{2,i+1} = \emptyset$ and that $F_{i+1} \, \se \, B \, \se \, F_{i+1} \cup U_{2,i+1} \se U_2$. 
\end{enumerate}

\noindent As $F = F_{\tau}$ and $C_{2} = U_{2,\tau}$ (see line~\ref{sunset:define-2}), this together with~\eqref{eq:basic-property-sunset} shows that item (3) holds.

For items (1) and (2), we have two possible outcomes:
either the algorithm satisfies the if-conditional in line~\ref{sunset:if}, and hence executes \sumrise$(s,A,F)$; or it satisfies the else-conditional in line~\ref{sunset:else} and sets $S= \emptyset$ and $(C_{0},C_{1}) = (U_{0},U_{1})$.
In the first case, since $|A| \ge s$, Claim~\ref{claim:basic-property-sumrise} shows that \sumrise$(s,A,F)$ is executed without error and items (1) and (2) are satisfied, with $|S| = s$; in the second case items (1) and (2) follow immediately, with $|S| = 0$.
\end{proof}

\begin{claim}\label{claim:basic-property-sunset}

Let $\delta>0$ and $s, \lambda, m \in \mathbb{N}_{\ge 1}$ and suppose that $(A,B)$ and $(A',B')$ belong to $\I_m(s,\lambda)$.
If $\sunset_{\delta}(s,\lambda,A,B)$ and $\sunset_{\delta}(s,\lambda,A',B')$ output the same fingerprint, then they output the same triple $(C_{0},C_{1},C_{2})$.

\end{claim}

\begin{proof}
Let $(F_i,U_{2,i})_{i=0}^{\tau}$ and $(F_i',U_{2,i}')_{i=0}^{\tau'}$ be the sequences generated by $\sunset_{\delta} \,(s,\lambda,A,B)$ and $\sunset_{\delta}(s,\lambda,A',B')$ in the while loop in lines~\ref{sunset:while-loop-start} to~\ref{sunset:increase}, respectively. 
Similarly, we denote by $(S,F)$, $(C_0,C_1,C_2)$ and $(S',F')$, $(C_0',C_1',C_2')$ the outputs of $\sunset_{\delta}(s,\lambda,A,B)$ and $\sunset_{\delta}(s,\lambda,A',B')$, respectively.

First, we show that if $(S,F) = (S',F')$, then the while loop is executed exactly in the same way for both pairs $(A,B)$ and $(A',B')$.
Suppose by contradiction that this is not the case and let $i$ be the first loop iteration for which $\sunset_{\delta}(s,\lambda,A,B)$ and $\sunset_{\delta}(s,\lambda,A',B')$ differ.
As all iterations before $i$ coincide in $\sunset_{\delta}(s,\lambda,A,B)$ and $\sunset_{\delta}(s,\lambda,A',B')$, we have
\[F_{i-1}=F_{i-1}' \qquad \text{and} \qquad U_{2,i-1}= U_{2,i-1}'.\]

Without loss of generality, suppose that at the $i$th iteration,
$\sunset_{\delta}(s,\lambda,A,B)$ executes the if-conditional in line~\ref{sunset:u-in-B} and $\sunset_{\delta}(s,\lambda,A',B')$ executes the else-conditional in line~\ref{sunset:u-not-in-B}.
Then, we have $F_i = F_{i-1} \cup \{u\}$ and $F_i' = F_{i-1}'$, where $u \coloneqq \sigma(\HH_{i-1})$.
Now, we note that the sequences of fingerprints $(F_i)_{i=0}^{\tau}$ and $(F_i')_{i=0}^{\tau'}$ are sequences of chains.
That is,
\[F_0 \se F_1  \se \cdots \se F_{\tau} = F
\qquad \text{and} \qquad
F_0' \se F_1'  \se \cdots \se F_{\tau'}' = F',
\]
and hence $u \in  F_{\tau} = F = F'$.
By Claim~\ref{claim:2nd-container-fingerprint}, we have $F' \se B'$, and hence $u \in B'$.
This is a contradiction, as $\sunset_{\delta}(s,\lambda,A',B')$ executed the else-conditional in line~\ref{sunset:not-add-to-F} because $u \notin B'$.
Once the while loop is executed exactly in the same way in both $\sunset_{\delta}(s,\lambda,A,B)$ and $\sunset_{\delta}(s,\lambda,A',B')$, 
it follows from line~\ref{sunset:define-2} that $C_2 = C_2'$.

Now, let $\HH = \HH(U_{0},U_{1},U_{2,\tau})$ and $\HH' = \HH(U_{0},U_{1},U_{2,\tau'}')$.
As $U_{2,\tau'}' = U_{2,\tau}$, we have $\HH = \HH'$.
It follows from lines~\ref{sunset:if} to~\ref{sunset:else} and Claim~\ref{claim:container-equal} that $C_{0}=C_{0}'$ and $C_{1}=C_{1}'$. This proves our claim.
\end{proof}

Now we are ready to prove Lemma~\ref{lemma:container-shrinking}.

\subsection*{Proof of Lemma~\ref{lemma:container-shrinking}}\label{sec:proof-container-shrinking}

For simplicity, set $\I_m(s,\lambda) = \I_{m}(s,\lambda, U_{0},U_{1},U_{2})$.
Fix a pair $(A,B) \in \I_m(s,\lambda)$ and let $(S,F)$ and $(C_{0},C_1,C_2)$ be the output of $\sunset_{\delta}(s,\lambda,A,B)$.
For simplicity, we set $\HH\coloneqq\HH(U_{0},U_{1},C_2)$ and $H\coloneqq H^{F}(U_{0},U_{1})$.
Before defining what is the container triple for $((A+B)^c,A,B)$, we shall first prove the following claim.

\begin{claim}\label{claim:dense-graph}
If $e(\HH) \ge \delta|U_{1}
||C_2|$, then $e(H) \ge \delta|F||U_{1}|$.
\end{claim}

\begin{proof}[Proof of Claim~\ref{claim:dense-graph}]
For each $u\in{F}$, define
$E_u\coloneqq\{\{u+v,v\}: u+v\in{U_{0}}\text{ and }v\in{U_{1}}\}$. We have
\[E(H) = \bigcup_{u \in F} E_u.
\]
Note that $|E_u|$ is exactly the number of hyperedges containing $u$ at the moment that $u$ is added to the fingerprint and that these sets $E_u$, $u\in{F}$, are pairwise disjoint.
Since $e(\HH) \ge \delta|U_{1}
||C_2|$, 
the maximum degree of a vertex of $\HH$ in $C_2$ is at least $\delta|U_{1}|$, and hence $|E_u| \ge \delta|U_{1}|$ for every $u \in F$.
Then, it follows that $e(H) \ge \delta|F||U_{1}|$.
\end{proof}

If $e(\HH) \ge \delta |U_{1}||C_2|$ or $e(H) \ge \delta|F||U_{1}|$, then $\sunset_{\delta}(s,\lambda,A,B)$ runs \sumrise$(s,A,F)$.
Since $(A,F) \in \K_{m}\big(s,\lambda\big)$, $ \lambda s \ge m$ and $|U_1| \ge m/4$, it follows from Claim~\ref{claim:dense-graph} and Lemma~\ref{lemma:graph-shrinking} that 
\begin{align}\label{eq:container-shrinking-11}
|C_{0}| \le |U_{0}|-\frac{\delta m}{6} \qquad \text{or} \qquad |S|+|C_1| \le \left (1-\frac{\delta}{6}\right)|U_{1}|.
\end{align}
On the other hand, if $e(\HH) < \delta |U_{1}||C_2|$ and $e(H) < \delta |F||U_1|$, then $\sunset_{\delta}(s,\lambda,A,B)$ executes the else-conditional in line~\ref{sunset:else} and it sets $S = \emptyset$, $C_0=U_0$ and $C_1 = U_{1}$. Namely, $\mathcal{H}=\mathcal{H}(C_0,C_1,C_2)$ and $H=H^F(C_0,C_1)$. Since $$E\left(\mathcal{H}(C_0,C_1,C_2\cup F)\right)=E\left(\mathcal{H}(C_0,C_1,C_2)\right)\,\dot\cup\, E\left(\mathcal{H}(C_0,C_1,F)\right)$$ and $e\left(\mathcal{H}(C_0,C_1,F)\right)=e\left(H^F(C_0,C_1)\right)$, the number of edges in the hypergraph $\HH(C_0,C_1,C_2 \cup F)$ is at most
\begin{align}\label{eq:container-shrinking-12}	
e\left(\HH(C_0,C_1,C_2 \cup F)\right) = e(\HH) +  e(H) < \delta |U_{1}|(|C_2|+|F|).
\end{align}

Finally, we set $$(D_0,D_1,D_2) = (C_0,C_1\cup S,C_2 \cup F)$$
to be the container triple for $((A+B)^c,A,B)$.
By Claim~\ref{claim:2nd-container-fingerprint} we have
\[ ((A+B)^c,A,B) \se (D_0,D_1,D_2) \se (U_0,U_1,U_2).\]
Now, we let $\D$ be the collection of triples $(D_0,D_1,D_2)$ when $(A,B)$ runs over $\I_m(s,\lambda)$.
By~\eqref{eq:container-shrinking-11} and~\eqref{eq:container-shrinking-12}, each of the triples in $\D$ satisfies at least one of the items (1) and (2).
By Claim~\ref{claim:basic-property-sunset}, for each pair of fingerprints $(S,F)$ there is a unique container triple $(C_0,C_1 \cup S,C_2 \cup F)$ associated with it.
Moreover, since we have $|F| = \lambda$ and $|S| \in \{0,s\}$ by Claim~\ref{claim:2nd-container-fingerprint}, it follows that
\[|\D| \le \left ( \binom{n}{s} +1 \right )\binom{n}{\lambda}.\]
This concludes our proof.\qed

\section{Counting sets with small sumset}\label{sec:counting}

This section is dedicated to proving Theorems~\ref{thm:symmetric-refined},~\ref{thm:unrefined-counting},~\ref{thm:asymmetric-counting} and~\ref{thm:unrefined-counting-part-2}.

\subsection*{Proof of Theorem~\ref{thm:symmetric-refined}}

	Set $\eps(n) = \left ( \frac{\sqrt{m}\log n}{s} \right )^{1/3}$. Since $m \ll \frac{s^2}{(\log n)^2}$, it holds that $0<\epsilon<1/4$.
	Recall that $\F_Y(m,s)$ denotes the family of sets $A \se Y$ such that $\lvert{A}\rvert=s$ and $|A+A| \le m$.
	As 
	\[ s = \eps^{-3}\sqrt{m}\log n \ge \lceil (1+\eps^{-2})\sqrt{m}\, \rceil,\]
	one can apply Theorem~\ref{thm:simple-containers} with $A=B$ to obtain a collection $\C$ of size at most $n^{2^{7}\eps^{-2}\sqrt{m}} = e^{O(\eps s)}$ (by choice of $\eps$) such that for every $A \in \F_Y(m,s)$ there exists $(C_0,C_1,C_2) \in \C$ such that $A \se C_1 \cap C_2$ and $|C_1|+|C_2| \le (1+2\eps)(m+\beta)$, where $\beta = \beta_G((1+4\eps)m)$.
	As $|C_1 \cap C_2| \le (|C_1|+|C_2|)/2$, it follows from these properties that
	\begin{align*}
		| \F_Y(m,s)| & \le \sum \limits_{(C_0,C_1,C_2) \in \C} \binom{|C_1 \cap C_2|}{s} \le e^{O(\eps s)} \cdot \mychoose{(1+2\eps)(m + \beta)/2}{s}.
	\end{align*}
	Since $s \le (m+\beta)/2$ and $8/s < 2\eps < 1/2$, it follows from the second part of Lemma~\ref{lemma:binomial-approx} that
	\begin{align*}
		\mychoose{(1+2\eps)(m + \beta)/2}{s} \le e^{8s\sqrt{2\eps}} \binom{\frac{m+\beta}{2}}{s}.
	\end{align*}
	As $m \ll \frac{s^2}{(\log n)^2}$, we have $\eps \ll 1$, which concludes the proof.\qed

\subsection*{Proof of Theorem~\ref{thm:unrefined-counting}}
	Let $0<\eps(n)<1/4$ to be chosen later and set $\alpha(n) = \left ( \frac{\sqrt{m}}{\log n} \right )^{1/2}$.
	Let $\F_1$ be the collection of sets $A$ in $\F_Y(m)$ with $|A| \ge \big \lceil (1+\alpha)\sqrt{m}\big \rceil$ and let $\F_2 = \F_Y(m) \setminus \F_1$.
	By Lemma~\ref{lemma:first-container} applied with $s = \lambda = \big \lceil (1+\alpha)\sqrt{m}\big \rceil$ and $\delta = \alpha^{-1}$, we have
	\begin{align*}
		|\F_1| \le \binom{n}{\lambda}^2 \cdot 2^{(1+\alpha^{-1})m} \le
		2^{(2+3\alpha)\sqrt{m}\log n + m -1}.
	\end{align*}
	In the last inequality, we used that $\alpha^{-1}m = \alpha\sqrt{m}\log{n}$ and we bounded the binomial coefficient $\mychoose{n}{\lambda}$ by $n^{\lambda}/2$.
	Since $|\F_2| \le \binom{n}{\le \lambda} \le n^{2(1+\alpha)\sqrt{m}}$, we have
	\begin{align*}
		|\F_Y(m)| = |\F_1| + |\F_2| \le 2^{(2+3\alpha)\sqrt{m}\log n + m}.
	\end{align*}
	This gives us our bound for $m \le (\log n)^2$.

	For now on, we assume that $m > (\log n)^2$ and use Theorem~\ref{thm:simple-containers} instead. Set $\eps=\alpha^{-2/3}/4\in\left(0,1/4\right)$. Let $t = \big \lceil (1+\eps^{-2})\sqrt{m}\big \rceil$ and let $\C$ be the collection given by Theorem~\ref{thm:simple-containers}.
    We have $|\C| \le n^{2^{7}\eps^{-2}\sqrt{m}}$
    and for each $A \in \F_Y(m)$ with size at least $t$ there exists a container pair $(C_1,C_2)$ such that $A \se C_1 \cap C_2$ and $|C_1|+|C_2| \le (1+2\eps)(m+\beta)$, where $\beta=\beta_G((1+4\eps)m)=\beta_{G}\left(m+\alpha^{-2/3}m\right)$.
	In particular, $C_1\cap C_2$ has size at most $(1+2\eps)(m + \beta)/2$.

	Let $\F_3$ be the collection of sets $A$ in $\F_Y(m)$ with size at least $t$.
	By Theorem~\ref{thm:simple-containers}, it follows that 
	\begin{align}\label{eq:size-F3}
		|\F_3| \le 2^{2^7\eps^{-2}\sqrt{m}\log n + (1+2\eps)(m + \beta)/2}.
	\end{align} 
	By using that $\beta = \beta_G((1+4\eps)m) \le 2m$ and $\eps^{-2}\sqrt{m}\log n = 2^{6}\eps m$, it follows from~\eqref{eq:size-F3} that
	\begin{align*}
		|\F_3| \le 2^{\frac{m+\beta}{2} + 2^{14} \eps m}.
	\end{align*}
	Now denote by $\F_4$ the collection of sets $A$ in $\F_Y(m)$ with size at most $t$. We have $|\F_4| \le \binom{n}{\le t} \le n^{t} \le 2^{2^{7}\eps m}$, and hence
	\begin{align*}
		|\F_Y(m)| = |\F_3| + |\F_4| \le 2^{\frac{m+\beta}{2} + 2^{15} \eps m}.
	\end{align*}
	The theorem follows by noting that $\eps = \alpha^{-2/3}/4$, and hence $2^{15} \eps = 2^{13} \alpha^{-2/3}$.\qed

\subsection*{Proof of Theorem~\ref{thm:asymmetric-counting}}
    Suppose without loss of generality that $s \le s'$. Set 
	\[\eps(n) = \left ( \frac{\sqrt{m}\log n}{s} \right )^{1/3} < 1/4.\]
	Recall that $\F_Y(m,s,s')$ denotes the family of pairs $(A,B)\in \binom{Y}{s} \times \binom{Y}{s'}$ for which we have $|A+B| \le m$.
	As $s = \eps^{-3}\sqrt{m}\log n \ge \lceil (1+\eps^{-2})\sqrt{m} \, \rceil $, by Theorem~\ref{thm:simple-containers} there exists a collection $\C \se 2^Y \times 2^Y$ with the following properties:
	\begin{enumerate}
		\item [(a)] $|\C| \le n^{2^{7}\eps^{-2}\sqrt{m}} \le e^{O(\eps s)}$;
        \vspace{1mm}
		\item [(b)] $|C|+|C'|\le (1+2\eps)(m + \beta)$ for every $(C,C') \in \C$;
        \vspace{1mm}
		\item [(c)] For every $(A,B) \in \F_Y(m,s,s')$ there exists $(C,C') \in \C$ such that $A \se C$ and $B \se C'$.
	\end{enumerate}

	It follows from items (b) and (c) and from Lemma~\ref{lemma:binomial-asymptotic-max} that
	\begin{align*}
		| \F_Y(m,s,s')| &\le \sum \limits_{(C,C') \in \C} \binom{|C|}{s} \binom{|C'|}{s'} 
		\le O(m^2)\cdot |\C| \mychoose{\frac{(1+2\eps)(m + \beta)s}{s+s'}}{s} \mychoose{\frac{(1+2\eps)(m + \beta)s'}{s+s'}}{s'}.
	\end{align*}
	Since $s+s' \le m+\beta$, $8/s < 2\eps < 1/2$, $s' = \Theta(s)$ and $m \le s^2 = 2^{o(s)}$,
	it follows from the inequality above together with item (a) and the second part of Lemma~\ref{lemma:binomial-approx} that
	\begin{align*}
		| \F_Y(m,s,s')| \le e^{O(\eps s +\sqrt{\eps}s)+o(s)} \cdot \mychoose{\frac{s(m+\beta)}{s+s'}}{s} \mychoose{\frac{s'(m+\beta)}{s+s'}}{s'}.
	\end{align*}
	As $m \ll \frac{s^2}{(\log n)^2}$, we have $\eps \ll 1$, and hence this concludes the proof.\qed

\subsection*{Proof of Theorem~\ref{thm:unrefined-counting-part-2}}
\linespread{}{
	Let $\delta(n) \in \mathbb{R}_{> 0}$ and set $s = \lceil \sqrt{m}\, \rceil$ and $\lambda = \lceil (1+\delta^{-1})\sqrt{m}\, \rceil$.
	Let $\F_1$ be the collection of pairs $(A,B)$ in $\F^{\text{asym}}_Y(m)$ with $|A| \ge s$ and $|B| \ge \lambda$. Let $\C_{1}$ be the collection of containers given by Lemma~\ref{lemma:first-container} with $s$ and $\lambda$, $\C_{2}$ be the collection of containers obtained from Lemma~\ref{lemma:first-container} by swapping $s$ and $\lambda$.
	Then every pair $(A,B)\in\mathcal{F}_{1}$ is contained in a container pair $(C_1,C_2)\in\C_1\times\C_2$, hence, we have
	\begin{align}\label{eq:size-F1}
		| \F_{1}| &\le \sum \limits_{(C_1,C_2) \in \C_1\times\C_2} 2^{|C_1|+|C_2|} \le \left(\binom{n}{s} \binom{n}{\lambda}\right)^2 2^{2(1+\delta)m} \le n^{(4+2\delta^{-1})\sqrt{m}} 2^{2(1+\delta)m-2}.
	\end{align}
	Set $\F_2 = \F^{\text{asym}}_Y(m) \setminus \F_1$.
	As $|A| \le |A+B| \le m$ for every pair $(A,B) \in \F_Y^{\text{asym}}(m)$, we have 
	\begin{align}\label{eq:size-F2}
		|\F_2| \le 2\mychoose{n}{\le \lambda} \binom{n}{\le m} \le  
		n^{(1+\delta^{-1})\sqrt{m}}\binom{n}{\le m}.
	\end{align}
	The first part of the theorem follows by combining~\eqref{eq:size-F1} with~\eqref{eq:size-F2}, setting $\delta=1$ and noting that $2^{4m} \le \binom{n}{m}$ if $m \le 2^{-4}n$.

	Let $\F_3$ be the collection of pairs $(A,B)$ such that $|A|, |B| \in [s,\lambda)$. Then, we have
	\begin{align}\label{eq:bound-on-F3-final}
		|\F_{Y}^{\mathrm{asym}}(m,\ge s)| \le 2|\F_1| + |\F_3|
		\le 2|\F_1| + n^{2(1+\delta^{-1})\sqrt{m}}.
	\end{align}
	It follows from~\eqref{eq:size-F1} and~\eqref{eq:bound-on-F3-final} that
	\begin{align}\label{eq:bound-asym-m-small}
		|\F_{Y}^{\mathrm{asym}}(m,\ge s)| \le n^{(4+2\delta^{-1})\sqrt{m}} 2^{2(1+\delta)m}.
	\end{align}
	One can also use Theorem~\ref{thm:simple-containers} to given another upper bound $|\F_1|$, and hence give another upper bound on $|\F_{Y}^{\mathrm{asym}}(m,\ge s)|$.
	Note that if $0<\delta<1/16$, then by setting $\eps=\sqrt{\delta}$ we have $|\F_1| \le n^{2^7 \delta^{-1}\sqrt{m}}2^{(1+2\sqrt{\delta})(m+\beta)}$, and hence
	\begin{align}\label{eq:bound-on-F1-large}
		|\F_{Y}^{\mathrm{asym}}(m,\ge s)| \le n^{2^7 \delta^{-1}\sqrt{m}}2^{(1+2\sqrt{\delta})(m+\beta)+1} + n^{2(1+\delta^{-1})\sqrt{m}}
	\end{align}
	whenever $0<\delta<1/16$.

	Now, let $\alpha(n) = \left(\frac{\sqrt{m}}{\log{n}}\right)^{1/2}$.
	If $m \le (\log n)^2$, then we set $\delta = \alpha^{-1}$ in~\eqref{eq:bound-asym-m-small}.
	Note that $\delta \ge 1$ and $\delta m = \delta^{-1}\sqrt{m} \log n $, we obtain
	\begin{align*}
		|\F_{Y}^{\mathrm{asym}}(m,\ge s)| \le n^{(4+2\delta^{-1})\sqrt{m}} 2^{4\delta m} \le n^{(4+6\delta^{-1})\sqrt{m}} \le n^{(4+6\alpha)\sqrt{m}}.
	\end{align*}
	If $m > (\log n)^2$, then we set $\delta = \alpha^{-4/3}/16 \in \left(0,1/16\right)$ in~\eqref{eq:bound-on-F1-large}.
	Since $\delta^{-1}\sqrt{m} \log n \le 2^6\sqrt{\delta} m$, we obtain
	\begin{align*}
		|\F_{Y}^{\mathrm{asym}}(m,\ge s)| \le 2^{(1+2\sqrt{\delta})(m+\beta)+1+2^{13}\sqrt{\delta} m} + n^{4\delta^{-1}\sqrt{m}},
	\end{align*}
	where $\beta=\beta_{G}\big((1+4\sqrt{\delta})m\big)=\beta_{G}\left(m+\alpha^{-2/3}m\right)$. As $4 \delta^{-1}\sqrt{m} \log{n} \le 2^{8} \sqrt{\delta} m$, it follows that
	\begin{align*}
    |\F_{Y}^{\mathrm{asym}}(m,\ge s)| \le 2^{(1+2^{15}\sqrt{\delta})(m+\beta)} =  2^{(1+2^{13}\alpha^{-2/3})(m+\beta)} .
	\end{align*}
 This concludes our proof.\qed}

\section{The typical structure of sets with small sumset}\label{sec:typical-structure}

In order to prove our typical structure result, we shall need a stability result of Campos, Coulson, Serra, and Wötzel~\cite{campos2023typical}.
Roughly speaking, it says that if we have two sets $D_{1},D_{2}$ of size around $m/2$ and a set $W$ of size $m$, then either $D_{1}$ and $D_{2}$ are close to an arithmetic progression or there are many pairs $(u,v) \in D_{1} \times D_{2}$ such that $u+v \notin W$.

\begin{lemma}{~\cite[Corollary 3.4]{campos2023typical}}\label{lemma:stability}
	Let $s_{1}\le{s_{2}}$ be positive integers, and $0<\eps \le 2^{-8}\left(\frac{s_{1}}{s_{1}+s_{2}}\right)^{2}$. 
	If $D_{1}, D_{2}, W \se \Z$ are such that 
	$(1-\eps)|W|\le |D_{1}|+|D_{2}|$ and 
	$|D_{i}| \le \left(\frac{s_{i}}{s_{1}+s_{2}}+2\sqrt{\eps}\right)|W|$ for all $i\in \{1,2\}$, then one of the following holds:
	\begin{enumerate}
		\item [$(a)$] there are at least $\eps^{2}|D_1| |D_2|$ pairs of $(u,v) \in D_{1} \times D_{2}$ such that $u+v \notin W$, or
        \vspace{1mm}
		\item [$(b)$] there exist arithmetic progressions $P_{1}, P_{2}$ with the same common difference and length
		\[ |P_{i}| \le \left(\frac{s_{i}}{s_{1}+s_{2}}+4\sqrt{\eps}\right)|W| \]
		containing all but at most $\eps |D_{i}|$ elements of $D_{i}$ for all $i\in \{1,2\}$.
	\end{enumerate}
	\end{lemma}

	Since the proof method of Theorems~\ref{thm:symmetric-typical-structure} and~\ref{thm:asymmetric-typical-structure} is exactly the same, we shall prove them together. In particular, throughout this section all logarithms are in base $e$.

\subsection*{Proof of Theorems~\ref{thm:symmetric-typical-structure} and~\ref{thm:asymmetric-typical-structure}}
	Suppose without loss of generality that $s\leq s'$. The case where $m = (1+o(1))(s+s')$ follows from a result of Lev and Smeliansky~\cite{lev1995addition}, which is an asymmetric version of Freiman's $3k-4$ theorem. Here we use the version stated in~\cite{tao2006additive}.
    \begin{lemma}{~\cite[Theorem 5.12]{tao2006additive}}
    \label{small}
    Let $A,B\subseteq\mathbb{Z}$ be finite sets such that $|A+B|\leq|A|+|B|+\min\{|A|,|B|\}-4$. Then $A$ is contained in an arithmetic progression of length at most $|A+B|-|B|+1$ and $B$ is contained in an arithmetic progression of length at most $|A+B|-|A|+1$, where both progressions have the same difference.
    \end{lemma}
	For any $(A,B) \in \binom{[n]}{s} \times \binom{[n]}{s'}$ satisfying $|A+B| \le m=(1+o(1))(s+s')$, since $s'=\Theta(s)$, it holds that $|A+B|\leq|A|+|B|+\min\{|A|,|B|\}-4$. Then, by Lemma~\ref{small} we have that there exist arithmetic progressions $P_1 \supseteq A$ and $P_2 \supseteq B$ with the same common difference and of sizes $|P_1| \leq|A+B|-|B|+1= (1 + o(1))s$ and $|P_2| \leq|A+B|-|A|+1= (1 + o(1))s'$.

    Therefore, one may assume from now on that there exists some absolute constant $\delta > 0$ such that $m \ge (1 + \delta)(s + s')$. As the same result applies when $s = s'$, for the symmetric case we may assume that $m \ge 2(1 + \delta)s$.

	For simplicity, let $\F^{\textrm{S}}$ be the family of sets $A \in \binom{[n]}{s}$ such that $|A+A| \le m$ and let $\F^{\textrm{A}}$ be the families of pairs $(A,B) \in \binom{[n]}{s} \times \binom{[n]}{s'}$ such that $|A+B| \le m$.
	First of all, recall that
	\[ |\F^{\textrm{S}}| = \Omega(1) \mychoose{m/2}{s}  \qquad \text{and} \qquad |\F^{\textrm{A}}| = \Omega(1)\mychoose{\frac{sm}{s+s'}}{s} \mychoose{\frac{s'm}{s+s'}}{s'}. \]
	Indeed, the bound on $|\F^{\textrm{A}}|$ comes from choosing $s$ elements among the first $\lfloor\frac{sm}{s+s'}\rfloor$ natural numbers and $s'$ elements among the first $\lfloor\frac{s'm}{s+s'}\rfloor$ natural numbers.
	Similarly, the bound on $|\F^{\textrm{S}}|$ comes from choosing $s$ elements among the first $\lfloor{m/2}\rfloor$ natural numbers.

	We shall use Theorem~\ref{thm:simple-containers} to show that most pairs in $\F^{\textrm{A}}$ and most elements in $\F^{\textrm{S}}$ are inside containers which are close to arithmetic progressions.
	This does not yet guarantee that the elements with bounded sumset are close to arithmetic progressions, but it is a first step towards this goal.
	Along the proof, we shall make several assumptions on which interval $\eps(n)$ must be.
	At the end, we shall show that the intersection of all these intervals is nonempty. That is, there exists an $\eps(n)$ which satisfies all these assumptions.

	Let $0<\eps(n)<1/4$ be such that
	\begin{align}\label{eq:eps-assumption-1}
		s \ge (1+\eps^{-2})\sqrt{m}
	\end{align}
 and let $\C$ be the collection given by Theorem~\ref{thm:simple-containers}.
	Then for every $(A,B) \in \F^{\textrm{A}}$ there exists $(C_{0},C_{1},C_{2}) \in \C$ such that $A \se C_1$, $B \se C_2$ and $(A+B)^c \se C_0$. 
	Similarly, for every $A \in \F^{\textrm{S}}$ there exists $(C_{0},C_{1},C_{2}) \in \C$ such that $A \se C_1 \cap C_2$ and $(A+A)^c \se C_0$. 
	Moreover,
	\begin{align}\label{eq:container-size-1}
		|\C| \le n^{2^{7}\eps^{-2}\sqrt{m}},
	\end{align}
	and the elements $(C_0,C_1,C_2)$ in $\C$ can be classified into two types:
	\begin{enumerate}
		\item [$(a)$] $|C_{1}| + |C_2| \le (1-\eps)m$;
        \vspace{1mm}
		\item [$(b)$] $|C_{1}| + |C_2| \le (1 + 2\eps)m$ and $e(\HH(C_0,C_1,C_2)) < \eps^2 |C_{1}||C_2|$.
	\end{enumerate}
	    
	Let $\C_{(a)}$ and $\C_{(b)}$ denote the families of containers of type $(a)$ and $(b)$, respectively.
	Similarly, let $\F^{\textrm{A}}_{(a)}$, $\F^{\textrm{A}}_{(b)}$ and $\F^{\textrm{S}}_{(a)}$, $\F^{\textrm{S}}_{(b)}$ be the families of pairs $(A,B) \in \F^{\textrm{A}}$ and the families of sets $A \in \F^{\textrm{S}}$ coming from containers of type $(a)$ and $(b)$, respectively.
	Then, by Lemma~\ref{lemma:binomial-asymptotic-max} we have
	\begin{align*}
		| \F^{\textrm{A}}_{(a)}| \le \sum \limits_{(C_0,C_1,C_2) \in \C_{(a)}} \binom{|C_1|}{s} \binom{|C_2|}{s'}  \le 4m^2|\C_{(a)}| \cdot \mychoose{\frac{(1-\eps)sm}{s+s'}}{s} \mychoose{\frac{(1-\eps)s'm}{s+s'}}{s'}.
	\end{align*}
 
	Assuming that 
	\begin{align}\label{eq:eps-assumption-4}
		  \log n \ll \eps s,
	\end{align}
    it follows from the first part of Lemma~\ref{lemma:binomial-approx} and \eqref{eq:container-size-1} that
 \begin{align}\label{eq:bound-F-A}
     | \F^{\textrm{A}}_{(a)}| \le e^{2^7 \eps^{-2}\sqrt{m}\log n - \eps (s+s')+o(\eps s)} \cdot \mychoose{\frac{sm}{s+s'}}{s} \mychoose{\frac{s'm}{s+s'}}{s'}.
 \end{align}
	Similarly, as $|C_1|+|C_2| \le (1-\eps)m$ implies that $|C_1 \cap C_2| \le (1-\eps)m/2$, we obtain
	\begin{align}\label{eq:bound-F-S}
		| \F^{\textrm{S}}_{(a)}| &\le |\C_{(a)}| \mychoose{\frac{(1-\eps)m}{2}}{s}  \le e^{2^7 \eps^{-2}\sqrt{m}\log n - \eps s} \cdot \mychoose{m/2}{s}.
	\end{align}
	In the last inequality, we used~\eqref{eq:container-size-1} and the first part of Lemma~\ref{lemma:binomial-approx}.

	Now, by taking $\eps$ such that 
	\begin{align}\label{eq:eps-assumption-2}
		2^{10}\eps^{-2} \sqrt{m} \log n \le 2\eps s \le \eps(s+s'),
	\end{align}
	it follows from~\eqref{eq:bound-F-A} that
	\begin{align}\label{eq:size-F-a}
		| \F^{\textrm{A}}_{(a)}| \le e^{-\eps (s+s')/2}\mychoose{\frac{sm}{s+s'}}{s} \mychoose{\frac{s'm}{s+s'}}{s'} \le e^{-\eps (s+s')/3}| \F^{\textrm{A}}|.
	\end{align}
	Similarly, it follows from~\eqref{eq:bound-F-S} that
	\begin{align}\label{eq:size-F-S-a}
		| \F^{\textrm{S}}_{(a)}| \le e^{-\eps s/2}\mychoose{m/2}{s}  \le e^{-\eps s/3}| \F^{\textrm{S}}|.
	\end{align}
	We shall break the containers of type $(b)$ into the following other types:
	\begin{enumerate}
		\item [$(b1)$] $ |C_{i}| > \left(\frac{1}{2}+2\sqrt{\eps}\right)m$ for some $i \in \{1,2\}$;
		\vspace{1mm}
        \item [$(b2)$] $|C_{i}| \le \left(\frac{1}{2}+2\sqrt{\eps}\right)m$ for all $i \in \{1,2\}$ and $e(\HH) < \eps^2 |C_{1}||C_2|$, where $\HH = \HH(C_0,C_1,C_2)$;
		\vspace{1mm}
        \item [$(b1')$] $|C_{1}| > \left(\frac{s}{s+s'}+2\sqrt{\eps}\right)m$ or $|C_{2}| > \left(\frac{s'}{s+s'}+2\sqrt{\eps}\right)m$;
		\vspace{1mm}
        \item [$(b2')$] $|C_{1}| \le \left(\frac{s}{s+s'}+2\sqrt{\eps}\right)m$, $|C_{2}| \le \left(\frac{s'}{s+s'}+2\sqrt{\eps}\right)m$
		 and $e(\HH) < \eps^2 |C_{1}||C_2|$, where $\HH = \HH(C_0,C_1,C_2)$.
	\end{enumerate}
	Similarly as before, we let $\C_{(b1)}$, $\C_{(b2)}$,  $\C_{(b1')}$ and $\C_{(b2')}$ be the families of containers in $\C_{(b)}$ which are of type $(b1)$, $(b2)$, $(b1')$ and $(b2')$, respectively.
	Observe that $\C_{(b1)} \cup \C_{(b2)} = \C_{(b)} = \C_{(b1')} \cup \C_{(b2')}$.
	Now, we let $\F_{(b1)}^{\textrm{S}}$ and $\F_{(b2)}^{\textrm{S}}$ be the families of sets $A \in \F_{(b)}^{\textrm{S}}$ coming from containers of type $(b1)$ and $(b2)$, respectively.
	Similarly, we let $\F_{(b1')}^{\textrm{A}}$ and $\F_{(b2')}^{\textrm{A}}$ be the families of pairs $(A,B) \in \F_{(b)}^{\textrm{A}}$ which come from containers of type $(b1')$ and $(b2')$, respectively.

	By taking 
	\begin{align}\label{eq:eps-assumption-3}
		\dfrac{2^{10} s^2}{m^2(s+s')^2} \le \eps \le \dfrac{\delta^2 s^2}{2^{10}(s+s')^2},
	\end{align}
        and using that $s \gg 1$ and that $2\sqrt{\eps}s\ge1$ (implied by~\eqref{eq:eps-assumption-4}),
	it follows from Lemma~\ref{lemma:binomial-approx-4} that
	\begin{align}\label{eq:bound-F-B-1-prime}
		|\F_{(b1')}^{\textrm{A}}| \le  \sum \limits_{(C_0,C_1,C_2) \in \C_{(b1')}} \binom{|C_1|}{s} \binom{|C_2|}{s'} \le 
		|\C_{(b1)}| \cdot e^{-\eps(s+s')}\binom{\frac{sm}{s+s'}}{s} \binom{\frac{s'm}{s+s'}}{s'}.
	\end{align}
	By using the bound in~\eqref{eq:container-size-1} on the number of containers and the bound on $\eps$ in~\eqref{eq:eps-assumption-2}, it follows from~\eqref{eq:bound-F-B-1-prime} that 
	\begin{align}\label{eq:size-F-b-1}
		| \F_{(b1')}^{\textrm{A}}| \le e^{-\eps (s+s')/2}\mychoose{\frac{sm}{s+s'}}{s} \mychoose{\frac{s'm}{s+s'}}{s'} \le e^{-\eps (s+s')/3} |\F^{\textrm{A}}|.
	\end{align}

	Similarly, if $(C_0,C_1,C_2) \in \C_{(b1)}$, then as $\C_{(b1)} \se \C_{(b)}$ we have 
	\begin{align}\label{eq:bound-C-i-sym}
	    |C_{i}| \le \left(\frac{1}{2}+2\eps - 2\sqrt{\eps}\right)m \le \frac{(1-\eps)m}{2} 
	\end{align}
	for some $i \in \{1,2\}$ (as long as $\eps < 16/25$).
	By using~\eqref{eq:container-size-1},~\eqref{eq:eps-assumption-2},~\eqref{eq:bound-C-i-sym} and the first part of Lemma~\ref{lemma:binomial-approx}, we obtain
	\begin{align}\label{eq:size-F-S-b-1}
		|\F_{(b1)}^{\textrm{S}}| \le  |\C_{(b1)}|\mychoose{\frac{(1-\eps)m}{2}}{s} \le  e^{2^7 \eps^{-2}\sqrt{m}\log n - \eps s} \cdot \mychoose{m/2}{s} \le e^{-\eps s/2}\mychoose{m/2}{s}  \le e^{-\eps s/3}| \F^{\textrm{S}}|.
	\end{align}

	Since we are assuming that $\eps s \gg \log n$ (see~\eqref{eq:eps-assumption-4}),   it follows from~\eqref{eq:size-F-a} and~\eqref{eq:size-F-b-1} that if we select a pair $(A,B) \in \F^{\textrm{A}}$ uniformly at random, then with high probability it will come from a container of type $(b2')$.
	Similarly, it follows from~\eqref{eq:size-F-S-a} and~\eqref{eq:size-F-S-b-1} that if we select a set $A \in \F^{\textrm{S}}$ uniformly at random, then with high probability it will come from a container of type $(b2)$.

	Note that the bound \eqref{eq:eps-assumption-3} on $\eps$ implies that $\eps<2^{-8}\left(\frac{s}{s+s'}\right)^{2}$.
    By Lemma~\ref{lemma:stability}, we have that for each $(C_0,C_1,C_2) \in \C_{(b2)} \cup \C_{(b2')}$ there exist arithmetic progressions $P_{1}$ and $P_{2}$ with the same common difference
	and sets $T_i \se C_i$ such that $|T_i| \le \eps |C_i|$ and $C_{i} \se P_i \cup T_i$ for all $i\in \{1,2\}$.
	We refer to $(T_1,T_2)$ as the pair of exceptional sets associated with $(C_1,C_2)$.

	Let $A \in \F_{(b2)}^{\textrm{S}}$, let $(C_0,C_1,C_2)$ be its associated container triple and let $(T_1,T_2)$ be the pair of exceptional sets associated with $(C_1,C_2)$.
	We say that $A$ is \emph{bad} if $|A \cap T_1| \ge 2^8 \eps s$ and $|A \cap T_2| \ge 2^8 \eps s$, and we set $\F_{\text{bad}}^{\textrm{S}}$ to be the collection of all bad sets in $\F_{(b2)}^{\textrm{S}}$.
	Similarly, let $(A,B) \in \F_{(b2)}^{\textrm{A}}$, let $(C_0,C_1,C_2)$ be its associated container triple and let $(T_1, T_2)$ be the pair of exceptional sets associated with $(C_1,C_2)$.
	We say that $A$ is \emph{bad} if $|A \cap T_1| \ge 2^8\eps(s+s')$ and we say that $B$ is \emph{bad} if $|B \cap T_2| \ge 2^8\eps(s+s')$ and we set $\F_{\text{bad}}^{\textrm{A},1}$ and $\F_{\text{bad}}^{\textrm{A},2}$ to be the collections of pairs $(A,B)$ in $\F_{(b2)}^{\textrm{A}}$ such that $A$ is bad and $B$ is bad, respectively.

	As $s = \Theta(s')$, to conclude the proof it suffices to show that $|\F_{\text{bad}}^{\textrm{S}}|= o(|\F^{\textrm{S}}|)$, that $|\F_{\text{bad}}^{\textrm{A},1}| + |\F_{\text{bad}}^{\textrm{A},2}| = o(|\F^{\textrm{A}}|)$ and that there exists an $\eps \ll 1$ satisfying all the assumptions made in the proof.

	Let us first bound $|\F_{\text{bad}}^{\textrm{S}}|$. Observe that 
	\begin{align}\label{eq:bound-T-1-T-2}
		\min\{|T_1|, |T_2|\} \le \eps m
	\end{align}
	for every pair of exceptional sets $(T_1,T_2)$.
	Indeed, given that $0<\eps<1/4$ (which is implied by~\eqref{eq:eps-assumption-3}, for example), we have $|C_1|+|C_2| \le 2m$ for every $(C_0,C_1,C_2) \in \C_{(b2)} \cup \C_{(b2')}$, and hence $\min\{|T_1|,|T_2|\} \le \frac{|T_1|+|T_2|}{2}\le \frac{\eps |C_1| + \eps |C_2|}{2} \le \eps m$. 
 
    It follows from~\eqref{eq:bound-T-1-T-2} that if $m < 2^8s$, then $\F_{\text{bad}}^{\textrm{S}} = \emptyset$. Thus, we may assume that $m \ge 2^8s$. For every set $A \in \F_{(b2)}^{\textrm{S}}$ there exists a container triple $(C_0,C_1,C_2)$ in $\C_{(b2)}$ such that $A \se C_1 \cap C_2$.
	As $|C_1 \cap C_2| \le (|C_1|+|C_2|)/2 \le (1/2+\eps)m$ and $\min\{|T_1|, |T_2|\} \le \eps m$, we obtain
	\begin{align}\label{eq:bound-F-bad-S}
		|\F_{\text{bad}}^{\textrm{S}}| 
		&\le \sum \limits_{(C_0,C_1,C_2) \in \C_{(b2)}} \sum \limits_{i = \lceil 2^8\eps s \rceil}^{s} \binom{|C_1 \cap C_2|}{s-i}\binom{\min\{\lvert{T_{1}}\rvert, \lvert{T_{2}}\rvert\}}{i} \nonumber\\
		& \le  |\C_{(b2)}| \sum \limits_{i = \lceil 2^8\eps s \rceil}^{s} \mychoose{\frac{(1+2\eps)m}{2}}{s-i}\binom{\eps m}{i}.
	\end{align}
	Since $s \le \frac{3}{4}\cdot\frac{(1+2\eps)m}{2}$ and $\eps s \gg \log n$ (see~\eqref{eq:eps-assumption-4}) one can use Lemma~\ref{lemma:binomial-bad-sets} with $\delta = 2\eps$ and $t = (1/2+\eps)m$ to bound the sum in~\eqref{eq:bound-F-bad-S}.
	By using this, the bound on the number of containers in~\eqref{eq:container-size-1} and the last part of Lemma~\ref{lemma:binomial-approx} with $\delta = 1$, 
    we obtain
	\begin{align*}
		|\F_{\text{bad}}^{\textrm{S}}|
		\le e^{2^7 \eps^{-2}\sqrt{m}\log n-2^8\eps s} \mychoose{\frac{(1+2\eps)m}{2}}{s} \le e^{2^7 \eps^{-2}\sqrt{m}\log n-2^7\eps s} \mychoose{m/2}{s}.
	\end{align*}
	By using the bound on $\eps$ coming from the assumption~\eqref{eq:eps-assumption-2}, we obtain that $|\F_{\text{bad}}^{\textrm{S}}| \le e^{-\eps s}|\F^{\textrm{S}}|$.

	Now let us bound $\F_{\text{bad}}^{\textrm{A},1}$. The bound on $\F_{\text{bad}}^{\textrm{A},2}$ is completely analogous.
	Let $(A,B) \in \F_{\text{bad}}^{\textrm{A},1} \se \F_{(b2)}^{\textrm{A}}$ be a pair, let $(C_0,C_1,C_2)\in \C_{(b2)}$ be its associated container triple and let $(T_1,T_2)$ be the pair of exceptional sets associated with it.
	As $|T_1| \le \eps |C_1|$,
	we must have $|C_1| \ge 2^8(s+s')$, otherwise $(A,B)$ would not be in $\F_{\text{bad}}^{\textrm{A},1}$.
	Set $\D$ to be the collection of all pairs $(C_1, C_2)$ such that $(C_0,C_1,C_2) \in \C_{(b2')}$ and $|C_1| \ge 2^8(s+s')$.
	Then, we have
	\begin{align*}\label{eq:size-F-bad}
		|\F_{\text{bad}}^{\textrm{A},1}| \le \sum \limits_{(C_1,C_2) \in \D} \sum \limits_{i = \lceil 2^8\eps (s + s') \rceil }^{s} \binom{|C_1|}{s-i}\binom{|T_1|}{i} \binom{|C_2|}{s'}.
	\end{align*}

	As $s \le s+s' \le 3|C_1|/4$ for every $(C_1,C_2) \in \D$, $|T_1| \le \eps |C_1|$ and  $\log n \ll \eps s$ (see~\eqref{eq:eps-assumption-4}), by applying Lemma~\ref{lemma:binomial-bad-sets} with $t = |C_1|$, $\Gamma = 2^8(s+s')/s$ and $\delta = \eps$, we obtain
	\begin{align*}
		|\F_{\text{bad}}^{\textrm{A},1}| \le e^{-2^8\eps (s+s')}\sum \limits_{(C_1,C_2) \in \D} \mychoose{|C_1|}{s} \mychoose{|C_2|}{s'}.
	\end{align*}
	As $|C_1|+|C_2| \le (1+2\eps)m$ for every $(C_1,C_2) \in \D$, it follows from Lemma~\ref{lemma:binomial-asymptotic-max} combined with the bound on the number of containers given in \eqref{eq:container-size-1} and the assumption $\log n \ll \eps s$ that
	\begin{align*}
		|\F_{\text{bad}}^{\textrm{A},1}| 
		\le e^{2^7 \eps^{-2}\sqrt{m}\log n-2^8\eps (s + s')+o(\eps s)} \mychoose{\frac{(1+2\eps)sm}{s+s'}}{s} \mychoose{\frac{(1+2\eps)s'm}{s+s'}}{s'}.
	\end{align*}
	By using the last part of Lemma~\ref{lemma:binomial-approx} with $\delta =1$ and~\eqref{eq:eps-assumption-2}, we obtain
    \begin{align*}
        |\F_{\text{bad}}^{\textrm{A},1}| \le e^{2^7 \eps^{-2}\sqrt{m}\log n-2^7\eps (s + s')} \mychoose{\frac{sm}{s+s'}}{s} \mychoose{\frac{s'm}{s+s'}}{s'}
		 \le e^{-2^6\eps (s + s')} \mychoose{\frac{sm}{s+s'}}{s} \mychoose{\frac{s'm}{s+s'}}{s'} \le e^{-\eps (s + s')} |\F^{\textrm{A}}|.
    \end{align*}

	Now, we just need to show that there is an $\eps \ll 1$ which satisfies all the assumptions made in the proof, namely, \eqref{eq:eps-assumption-1}, \eqref{eq:eps-assumption-4}, \eqref{eq:eps-assumption-2} and \eqref{eq:eps-assumption-3}.
	Since $s = \Theta(s')$, it suffices to have
	\begin{align}
		\max \left \{\left(\dfrac{\sqrt{m}}{s}\right)^{1/2}, \dfrac{\log n}{s}, \left(\dfrac{\sqrt{m}\log n}{s}\right)^{1/3}, \dfrac{s^2}{m^2(s+s')^2} \right \} \ll \eps \ll 1.
	\end{align}
	Such $\eps$ exists if $s = \Theta(s')$ and $1 \le m \ll \frac{s^2}{(\log n)^{2}}$.\qed

\section*{Acknowledgments}
We would like to thank Marcelo Campos for helpful discussions. We also thank the referee for their careful reading and valuable comments which greatly improved the presentation of this paper.

\newpage
\appendix

\section{}\label{appendix:lower-bound}

\begin{lemma} If $6s \le 0.99m$ and $\frac{m+1}{2}\le 0.99n$, then
    \begin{align*}
        |\F_{[n]}(m,s)| = \Omega\left( \frac{n^2s}{m^2} \right) \mychoose{ \lfloor\frac{m+1}{2}\rfloor }{s}.
    \end{align*}
\end{lemma}

\begin{proof}
    For each $a \in \big \{\frac{3s}{2},\ldots, \lfloor \frac{m+1}{2} \rfloor \big \}$, let $\mathcal{T}_{a}$ denote the collection of $s$-element subsets $A$ satisfying $\min A = 0$, $\max A = a$, and $\text{gcd}(A)=1$.
    Notice that, aside from the minimum and maximum, the
    elements of every set $A \in \mathcal{T}_{a}$ are contained in $[a-1]$.
    Hence, an immediate upper bound is $|\mathcal{T}_{a}| \le \binom{a-1}{s-2}$.
    To establish a corresponding lower bound, we proceed as follows.
    First, choose two numbers $x,y$ from $[a-1]$ that are coprime.
    There are $\big ( \frac{6}{\pi^2} + o(1) \big ) \binom{a-1}{2}$ such pairs
    (see, for example, Theorem 332 in~\cite{Hardy-Wright}).
    Next, select the remaining $s-4$ elements from $[a-1]\setminus \{x,y\}$ to complete the $s$-set.
    Since every $s$-set is counted $\binom{s-2}{2}$ times by this procedure, we obtain
    \begin{align}\label{eq:size-of-T}
        |\mathcal{T}_{a}| \ge \dfrac{\big ( \frac{6}{\pi^2} + o(1) \big ) \binom{a-1}{2} \binom{a-3}{s-4}}{\binom{s-2}{2}} = \Omega(1) \binom{a-1}{s-2}.
    \end{align}
    This shows that, up to constant factors, the size of $\mathcal{T}_{a}$ is $\binom{a-1}{s-2}$.

    For any $t,d \in \mathbb{N}_{\ge 0}$ and $S \se \mathbb{N}_{\ge 0}$, we write $t+d \cdot S = \{t+dk: k \in S\}$.
    Note that if $A$ is a finite set containing $0$ and another nonzero element, then distinct pairs $(t_1,d_1)$ and $(t_2,d_2)$ yield distinct sets $t_1 + d_1 A$ and $t_2 + d_2 A$.
    For $a,n \in \mathbb{N}_{\ge1}$, define 
    \begin{align*}
        \mathcal{D}_{a,n} \coloneqq \Big \{ t+ d \cdot A \se [n]: A \in \mathcal{T}_a, \, t,d \in \mathbb{N}_{\ge 1}\Big \}.
    \end{align*}
    It follows that the size of $\mathcal{D}_{a,n}$ equals to 
    \begin{align}\label{eq:size-of-D}
        |\mathcal{D}_{a,n}| & = |\mathcal{T}_a| \cdot \# \{(t,d): t,d\ge 1 \text{ and }  t+da \in [n]\} \nonumber \\
        & = |\mathcal{T}_a| \cdot \sum_{t=1}^n \left \lfloor \dfrac{n-t}{a} \right \rfloor \ge |\mathcal{T}_a| \cdot \sum_{t=1}^{n-a} \dfrac{n-t}{2a}.
    \end{align}
    When $ 1 \le a \le 0.99n$, using the lower bound from~\eqref{eq:size-of-T} for $|\mathcal{T}_a|$ and evaluating the sum in~\eqref{eq:size-of-D}, we obtain 
    \begin{align}\label{eq:bound-on-D-2}
        |\mathcal{D}_{a,n}| \ge \dfrac{ |\mathcal{T}_a|}{2a} \left ( \binom{n}{2} - \binom{a}{2} \right) = \Omega (1) \dfrac{n^2}{a} \binom{a-1}{s-2} = \Omega (1) \dfrac{n^2}{s} \binom{a-2}{s-3}.
    \end{align}

    Now we claim that for distinct $a, a' \in \mathbb{N}_{\ge 1}$ and any $n \in \mathbb{N}_{\ge 1}$, we have $\mathcal{D}_{a,n} \cap \mathcal{D}_{a',n} = \emptyset$.
    Assume, for the sake of contradiction, that there exists an element $X \in \mathcal{D}_{a,n} \cap \mathcal{D}_{a',n}$.
    By definition of $\mathcal{D}_{a,n}$, there exist integers $t_1, d_1 \ge 1$ and a set $A \in \mathcal{T}_a$ such that 
    $X = t_1 + d_1 \cdot A.$
    Similarly, since $X \in \mathcal{D}_{a',n}$, there exist $t_2, d_2 \ge 1$ and a set $A' \in \mathcal{T}_{a'}$ with 
    $X = t_2 + d_2 \cdot A'.$
    Because $0 \in A \cap A'$, it follows immediately from
$t_1 + d_1 \cdot 0 = t_2 + d_2 \cdot 0,$
that $t_1 = t_2$. Hence, we deduce
$d_1 \cdot A = d_2 \cdot A'.$
Since each \(A \in \mathcal{T}_a\) and \(A' \in \mathcal{T}_{a'}\) satisfies \(\gcd(A) = \gcd(A') = 1\), we have
\[
\gcd(d_1 \cdot A) = d_1 \quad \text{and} \quad \gcd(d_2 \cdot A') = d_2.
\]
Thus, it must be that $d_1 = d_2$, and hence $A = A'$, which forces $a = a'$, contradicting our assumption that $a$ and $a'$ are distinct.

Every set $A \in \mathcal{D}_{a,n}$ satisfies $|A+A| \le 2a-1$, so in particular, for every $a \in \left[\lfloor \frac{m+1}{2} \rfloor\right]$ we have 
$$
\mathcal{D}_{a,n} \subseteq \mathcal{F}_{[n]}(m,s).
$$
Since the families $\mathcal{D}_{a,n}$ are pairwise disjoint for distinct $a$, we obtain from~\eqref{eq:bound-on-D-2} that if $\frac{m+1}{2} \le 0.99n$ then
\begin{align*}
        |\mathcal{F}_{[n]}(m,s)| \ge \sum_{a=\lceil 3s/2 \rceil }^{\lfloor \frac{m+1}{2} \rfloor}|\mathcal{D}_{a,n}| = \Omega \left ( \dfrac{n^2}{s} \right ) \sum \limits_{a=\lceil 3s/2 \rceil }^{\lfloor \frac{m+1}{2} \rfloor}\binom{a-2}{s-3} = \Omega \left ( \dfrac{n^2}{s} \right ) \mychoose{\lfloor \frac{m+1}{2} \rfloor -1}{s-2}.
    \end{align*}
    In the last inequality, we used that $0.99m \ge 6s$ and the Hockey-stick identity.
    The lemma follows by combining the identity $\frac{a}{b}\binom{a-1}{b-1}=\binom{a}{b}$ with the previous inequality.
\end{proof}

\section{}
\begin{lemma}\label{lemma:binomial-asymptotic-max}
	Let $1 \le s_1 \le s_2 \le m$ be integers with $s_1+s_2 \le m$. Then, we have
	\[ \max_{x \in \{0,\ldots,m\}} \binom{x}{s_1} \binom{m-x}{s_2} \le 4m^2 \binom{\frac{s_1m}{s_1+s_2}}{s_1}\binom{\frac{s_2m}{s_1+s_2}}{s_2}.\]
	Moreover, $\binom{x}{s_1} \binom{m-x}{s_2}$ is non-decreasing in the interval $\big\{0,\ldots,\big\lfloor \frac{s_1 m+s_1}{s_1+s_2} \big\rfloor\big\}$ and non-increasing in the interval $\big\{\big\lfloor \frac{s_1 m+s_1}{s_1+s_2} \big\rfloor,\ldots,m\big\}$.
\end{lemma}

\begin{proof}
	For simplicity, let $f: \mathbb{N}_{\ge0} \to \mathbb{N}_{\ge0}$ be the function given by
	\[ f(x) \coloneqq \binom{x}{s_1} \binom{m-x}{s_2},\]
	If $x \in \{0,\ldots,s_1-1\} \cup \{m-s_2+1,\ldots,m\}$, then $f(x) = 0$.
	Now, suppose that $x \in [s_1,m-s_2]$. Then,
	\begin{align*}
		\dfrac{f(x-1)}{f(x)} = \dfrac{x-s_1}{x} \cdot \dfrac{m-x+1}{m-x+1-s_2}.
	\end{align*}
	As $m-x+1-s_2>0$ for every $x \in [s_1,m-s_2]$, we have that the ratio above is at most 1 if and only if $ (x-s_1)(m-x+1) \le x(m-x+1-s_2)$, which is equivalent to  
	$$x \le \frac{s_1(m+1)}{s_1+s_2}.$$
	Now, set $t =\big\lfloor \frac{s_1 m+s_1}{s_1+s_2} \big\rfloor$.
	From the analysis it follows that $f$ is non-decreasing on $\{0,\ldots,t\}$ and non-increasing on $\{t,\ldots,m\}$. Therefore,
	\begin{align}\label{eq:binomial-max}
		\max_{x \in \{0,\ldots,m\}} f(x) = f(t) \le \binom{\frac{s_1m+s_1}{s_1+s_2}}{s_1}\binom{\frac{s_2m+s_2}{s_1+s_2}}{s_2}. 
	\end{align}
	Now, note that for every $\theta \in [0,1]$, $a \in \mathbb{R}_{>0}$ and $b \in \mathbb{N}_{\ge 0}$ with $b\le a$, we have
	\begin{align}\label{eq:binomial-approx}
		\binom{a+\theta}{b} \le \binom{a+1}{b} = \dfrac{a+1}{a+1-b}\binom{a}{b} \le 2a \binom{a}{b}.
	\end{align}
	As $1 \le s_1 \le s_2$ and $s_1+s_2 \le m$, the lemma follows from~\eqref{eq:binomial-max} and~\eqref{eq:binomial-approx}.
\end{proof}

\begin{lemma}\label{lemma:binomial-approx}	
	Let $a,b,d,\eps\in\mathbb{R}_{>0}$ and $s\in\mathbb{N}_{\ge1}$ be such that $s \le b \le a$, $s \le d$ and $8/d<\eps<1/2$. Then we have
	\[ 
		\binom{b}{s} \le \left(\dfrac{b}{a}\right)^s \binom{a}{s} 
		\qquad \text{and} \qquad
		\binom{(1+\eps)d}{s} \le e^{8\sqrt{\eps}s} \binom{d}{s}.
	\]
        Moreover, if $(1+\delta)s \le d$ with $\delta>0$, then $\binom{(1+\eps)d}{s} \le \left ( 1 + \eps + \frac{\eps}{\delta} \right )^s \binom{d}{s}$ for all $\eps>0$.
\end{lemma}

\begin{proof}	
	Note that
	\begin{align*}
		a^s s! \binom{b}{s} &= a^s \cdot b(b-1)(b-2)\cdots (b-s+1)\\
		&\le b^s \cdot a (a-1)\cdots(a-s+1)\\
		&= b^s s! \binom{a}{s},
	\end{align*}
	which proves the first part.

	For the second part, we divide the proof into two cases.
	The first case is when $3s \le 2d$. In this case, note that 
	$(1+\eps)d-i \le \big(1+3\eps\big)(d-i)$ as long as $3i \le 2d$, which implies that
	\begin{align}\label{eq:case-d-big}
		\binom{(1+\eps)d}{s} \le (1+3\eps)^s \binom{d}{s} \le e^{3\eps s} \binom{d}{s}.
	\end{align}

	The second case is when $3s > 2d$. In this case, we shall use that
	\begin{align}\label{eq:binomial-estimate-d-small}
		\binom{x+y}{s} \le \binom{\lceil{x+y}\rceil}{s} \binom{\lceil{x+y}\rceil}{\lceil{x+y}\rceil-\lfloor{x}\rfloor} \binom{x}{s} \le \left (\frac{e(x+y+1)}{y} \right )^{y+2} \binom{x}{s},
	\end{align} 
	which holds for all $x,y\in\mathbb{R}_{>0}$ with $x \ge s$.
	Plugging in $x = d$ and $y = \eps d$ in~\eqref{eq:binomial-estimate-d-small}, recall that $8/d<\eps<1/2$, we obtain
	\begin{align}\label{eq:case-d-small}
		\binom{(1+\eps)d}{s} \le \left (\frac{2e(1+\eps)d}{\eps d} \right )^{\eps d+2} \binom{d}{s} \le e^{5\eps d \ln\left(\frac{1}{\eps}\right)} \binom{d}{s}.
	\end{align}
	As $2d < 3s$ and $\eps \ln\left(\frac{1}{\eps}\right) \le \sqrt{\eps}$ for $\eps\in(0,1)$, it follows from~\eqref{eq:case-d-small} that $\binom{(1+\eps)d}{s} \le e^{8\sqrt{\eps}s} \binom{d}{s}$.

    For the last part, simply note that $(1+\eps)d-i \le \big(1+\eps+\frac{\eps}{\delta}\big)(d-i)$ as long as $(1+\delta)i \le d$.
\end{proof}

\begin{lemma}\label{lemma:binomial-approx-4}
	Let $s_1, s_2, m$ be positive integers, $\delta, \eps \in (0,1)$ be such that $(1+\delta)(s_1+s_2) \le m$, $s_1 + s_2 \ge 2^{5} \delta^{-1}$, $2\sqrt{\eps}\min\{s_{1}, s_{2}\}\geq1$ and
	\[ \dfrac{2^{10} \min\{s_{1}^{2}, s_2^2\}}{m^2(s_1+s_2)^2} \le \eps \le \dfrac{\delta^2 \min\{s_{1}^{2}, s_2^2\}}{2^{10}(s_1+s_2)^2}.\]
	Let $\D$ be the set of all pairs $(t,x) \in \mathbb{N}_{\ge 0} \times \mathbb{N}_{\ge 0}$ such that $s_1 + s_2 \le t \le (1+2\eps)m$ and $ \frac{s_1 m}{s_1+s_2} + 2\sqrt{\eps}m \le x \le t -s_2$. Then, we have
	\[\max_{(t,x) \in \D}\binom{x}{s_1} \binom{t-x}{s_2} \le e^{-\eps(s_1+s_2)}\binom{\frac{s_1m}{s_1+s_2}}{s_1} \binom{\frac{s_2m}{s_1+s_2}}{s_2}.\]
\end{lemma}

\begin{proof}
	For simplicity, set $m' = (1+2\eps)m$ and let $f: \D \to \mathbb{Z}$ be the function given by $f(t,x) = \binom{x}{s_1} \binom{t-x}{s_2}$.
	As $t \to \binom{t-x}{s_2}$ is increasing in the interval $[x+s_2,\infty)$, we have $f(t,x) \le f(m',x)$ for all $(t,x) \in \D$.
	Moreover, as $2\sqrt{\eps}\min\{s_{1}, s_{2}\}\geq1$, we have $\frac{s_1 m}{s_1+s_2} + 2\sqrt{\eps}m \ge \frac{s_1 m' + s_{1}}{s_1+s_2}$, and as $x \to f(\lfloor m' \rfloor,x)$ is non-increasing in the interval $\big[\big\lfloor \frac{s_1 m' + s_{1}}{s_1+s_2} \big\rfloor,m'\big]$ (see Lemma~\ref{lemma:binomial-asymptotic-max}), it follows that
	\begin{align}\label{eq:binomial-max-2}
		\max_{(t,x) \in \D} f(t,x) = f \left (m', \frac{s_1 m}{s_1+s_2} + 2\sqrt{\eps}m \right ).
	\end{align}
	Now, we use Lemma 4.2 of~\cite{campos2023typical}.
	This lemma states that if $2^5\delta^{-1} \le s_1+s_2$, $(1+\delta)(s_1+s_2)\le m$ and
	\[ \dfrac{2^{10} \min\{s_1^2, s_{2}^{2}\}}{m^2(s_1+s_2)^2} \le \eps \le \dfrac{\delta^2 \min\{s_1^2, s_{2}^{2}\}}{2^{10}(s_1+s_2)^2},\]
	then we have 
	\[ f \left (m', \frac{s_1 m}{s_1+s_2} + 2\sqrt{\eps}m \right ) \le  e^{-\eps(s_1+s_2)}\binom{\frac{s_1m}{s_1+s_2}}{s_1} \binom{\frac{s_2m}{s_1+s_2}}{s_2}.\qedhere\]
\end{proof}

\begin{lemma}\label{lemma:binomial-bad-sets}
	Let $\delta,\Gamma,t\in\mathbb{R}_{>0}$ and $s \in \mathbb{N}_{\ge 1}$ be such that $\Gamma \ge 2^5$, $s \le 3t/4$ and $s \le 2^{\Gamma \delta s}$.
    Then, we have
 \[\sum  \limits_{i = \lceil \Gamma \delta s \rceil}^{s} \binom{t}{s-i}\binom{\delta t}{i} \le 2^{-\Gamma\delta s} \binom{t}{s}.\]
\end{lemma}

\begin{proof}
  First, we claim that if $a,b\in\mathbb{R}_{>0}$ and $c,d\in\mathbb{N}_{\ge1}$ with $d \le c \le 3a/4$, then $\binom{a}{c-d} \binom{b}{d}  \le \left ( \frac{4bc}{ad}\right )^d \binom{a}{c}$.
  Indeed, note that
  \begin{align*}
      \binom{a}{c-d} \binom{b}{d}  = \frac{c-d+1}{a-c+d} \cdot \dfrac{b-d+1}{d}\binom{a}{c-d+1)} \binom{b}{d-1} \le \frac{cb}{(a-c)d} \binom{a}{c-d+1} \binom{b}{d-1}.
  \end{align*}
   Then, the inequality follows by noting that $c/(a-c) \le 4c/a$ and using induction on $d$.
    
    It follows that if $s \le 3t/4$, then we have
	\begin{align*}
		\sum  \limits_{i = \lceil \Gamma \delta s \rceil}^{s} \binom{t}{s-i}\binom{\delta t}{i} \le \sum \limits_{i = \lceil \Gamma \delta s \rceil}^{s} \left ( \dfrac{4\delta  s}{i} \right )^i \binom{t}{s}.
	\end{align*}
	As the function $x \to \left ( \frac{\alpha}{x} \right )^x$ is decreasing in the interval $[\alpha e^{-1},+\infty)$ and as $ 4e^{-1}\delta s \le \Gamma \delta s/2$, we have that the above sum is at most
	\[ s \left ( \dfrac{8}{\Gamma} \right)^{-\Gamma \delta s/2} \binom{t}{s} \le 2^{-\Gamma \delta s} \binom{t}{s}.\qedhere\]
\end{proof}

\end{document}